\documentclass[11pt]{amsart}

\usepackage{amsmath}
\usepackage{amssymb}
\usepackage{amsthm}
\usepackage{cancel}
\usepackage{tcolorbox}
\usepackage[margin=1.2in]{geometry}
\usepackage{tikz,tikz-cd}
\usetikzlibrary{external}
\usepackage{cases}
\usepackage[]{mdframed}
\usepackage{url}
\usepackage{hyperref}

\newcommand{\Z}{\mathbb{Z}}
\newcommand{\C}{\mathbb{C}}
\newcommand{\R}{\mathbb{R}}

\theoremstyle{definition}
\newtheorem{definition}{Definition}

\theoremstyle{remark}
\newtheorem{remark}{Remark}
\theoremstyle{plain}
\newtheorem{theorem}{Theorem}
\newtheorem{proposition}{Proposition}
\newtheorem{lemma}{Lemma}
\newtheorem{corollary}{Corollary}

\numberwithin{equation}{section}
\numberwithin{lemma}{section}
\numberwithin{definition}{section}
\numberwithin{proposition}{section}
\numberwithin{corollary}{section}
\numberwithin{theorem}{section}
\numberwithin{remark}{section}

\title{On the heat kernel of a Cayley graph of $\operatorname{PSL}_2\Z$}
\author{Anders Karlsson}
\address{Section de mathématiques, Université de Genève, rue du Conseil-Général 7-9, 1205 Genève, Suisse;  Matematiska institutionen, Uppsala universitet, Box 256, 751 05, Uppsala, Sweden}
\email{anders.karlsson@unige.ch}

\author{Kamila Kashaeva}
\address{Section de mathématiques, Université de Genève, rue du Conseil-Général 7-9, 1205 Genève, Suisse}
\email{kamila.kashaeva@unige.ch}
\date{}
\begin{document}

\maketitle

\begin{abstract}
In this paper, we obtain an explicit formula for the heat kernel on the Cayley graph of the modular group $\operatorname{PSL}_2\Z$, given by the presentation $\langle a,b\mid a^2=1,  b^3=1\rangle$. Our approach extends a method of Chung--Yau by observing that the Cayley graph strongly and regularly covers a weighted infinite line. We solve the spectral problem on this line to obtain an integral expression for its heat kernel, and then lift this to the Cayley graph using spectral transfer principles for strongly regular coverings. The explicit formula allows us to determine the Laplace spectrum, containing eigenvalues and continuous parts. As a by-product, we suggest a conjecture on the lower bound for the spectral gap of Cayley graphs of $\operatorname{PSL}_2\mathbb{F}_p$ with our generators, inspired by the analogy with Selberg's $1/4$-conjecture. Numerical evidence to this conjecture is provided for small primes.
\end{abstract}

\section{Introduction}
There are not that many examples of infinite graphs with explicitly known spectrum and heat kernel. For instance, Chung–Yau~\cite{MR1667452}, Cowling--Meda--Setti~\cite{MR1653343}, 
and Chinta–Jorgenson--Karlsson~\cite{MR3394421} provided explicit formulas for the heat kernel of the infinite $k$-regular tree, see also~\cite{MR5005075} which provides a general method for explicit heat kernels on infinite graphs. A wealth of examples of spectra of infinite graphs can be found in a recent paper by Grigorchuk--Nagnibeda--Pérez \cite{MR4458570} and references therein. 

In this paper, we contribute the example of a Cayley graph of the group $\operatorname{PSL}_2\Z$. More specifically, the Cayley graph $\Gamma$ associated to the following group presentation
\begin{equation}\label{group}
 G=\operatorname{PSL}_2\Z \cong C_2*C_3=\langle a,b\mid a^2=1,  b^3=1\rangle,
\end{equation}
with $\Gamma$ drawn in Fig.~\ref{fig:cayley}.
\begin{figure}[h]
\center
 \begin{tikzpicture}[scale=1]

\coordinate (0) at (0,0);
\coordinate (1_1) at (1,0.5);
\coordinate (1_2) at (1,-0.5);
\coordinate (2_1) at (2,1.2);
\coordinate (2_2) at (2,-1.2);

\begin{scope}[shift={(2,1.2)}]
\coordinate (3_1) at (1,0.5);
\coordinate (3_2) at (1,-0.5);
\coordinate (4_1) at (2,1.2);
\coordinate (4_2) at (2,-1);
\node at (2.5,1.2){$\dots$};
\node at (2.5,-1){$\dots$};
\end{scope}

\begin{scope}[shift={(2,-1.2)}]
\coordinate (3_3) at (1,0.5);
\coordinate (3_4) at (1,-0.5);
\coordinate (4_3) at (2,1);
\coordinate (4_4) at (2,-1.2);
\node at (2.5,1){$\dots$};
\node at (2.5,-1.2){$\dots$};
\end{scope}

\coordinate (-1) at (-1,0);
\begin{scope}[shift={(-1,0)}]
\coordinate (-2_1) at (-1,0.5);
\coordinate (-2_2) at (-1,-0.5);
\coordinate (-3_1) at (-2,1.2);
\coordinate (-3_2) at (-2,-1.2);
\end{scope}

\begin{scope}[shift={(-3,1.2)}]
\coordinate (-4_1) at (-1,0.5);
\coordinate (-4_2) at (-1,-0.5);
\coordinate (-5_1) at (-2,1.2);
\coordinate (-5_2) at (-2,-1);
\node at (-2.5,1.2){$\dots$};
\node at (-2.5,-1){$\dots$};
\end{scope}

\begin{scope}[shift={(-3,-1.2)}]
\coordinate (-4_3) at (-1,0.5);
\coordinate (-4_4) at (-1,-0.5);
\coordinate (-5_3) at (-2,1);
\coordinate (-5_4) at (-2,-1.2);
\node at (-2.5,1){$\dots$};
\node at (-2.5,-1.2){$\dots$};
\end{scope}

\draw[fill=black] (0) circle [radius=2pt];
\draw[fill=black] (1_1) circle [radius=2pt];
\draw[fill=black] (1_2) circle [radius=2pt];
\draw[fill=black] (2_1) circle [radius=2pt];
\draw[fill=black] (2_2) circle [radius=2pt];
\draw[fill=black] (3_1) circle [radius=2pt];
\draw[fill=black] (3_2) circle [radius=2pt];
\draw[fill=black] (3_3) circle [radius=2pt];
\draw[fill=black] (3_4) circle [radius=2pt];
\draw[fill=black] (4_1) circle [radius=2pt];
\draw[fill=black] (4_2) circle [radius=2pt];
\draw[fill=black] (4_3) circle [radius=2pt];
\draw[fill=black] (4_4) circle [radius=2pt];
\draw[fill=black] (-1) circle [radius=2pt];
\draw[fill=black] (-2_1) circle [radius=2pt];
\draw[fill=black] (-2_2) circle [radius=2pt];
\draw[fill=black] (-3_1) circle [radius=2pt];
\draw[fill=black] (-3_2) circle [radius=2pt];
\draw[fill=black] (-4_1) circle [radius=2pt];
\draw[fill=black] (-4_2) circle [radius=2pt];
\draw[fill=black] (-4_3) circle [radius=2pt];
\draw[fill=black] (-4_4) circle [radius=2pt];
\draw[fill=black] (-5_1) circle [radius=2pt];
\draw[fill=black] (-5_2) circle [radius=2pt];
\draw[fill=black] (-5_3) circle [radius=2pt];
\draw[fill=black] (-5_4) circle [radius=2pt];

\draw(0)--(1_1)--(1_2)--(0);
\draw (1_1) to[out=45,in= 195] (2_1); \draw (1_1) to[out=15,in= 225] (2_1);
\draw (1_2) to[out=-15,in=135] (2_2); \draw (1_2) to[out=-45,in= 165] (2_2);
\draw(2_1)--(3_1)--(3_2)--(2_1);
\draw(2_2)--(3_3)--(3_4)--(2_2);
\draw (3_1) to[out=45,in= 195] (4_1); \draw (3_1) to[out=15,in= 225] (4_1);
\draw (3_2) to[out=-15,in=135] (4_2); \draw (3_2) to[out=-45,in= 165] (4_2);
\draw (3_3) to[out=45,in= 195] (4_3); \draw (3_3) to[out=15,in= 225] (4_3);
\draw (3_4) to[out=-15,in=135] (4_4); \draw (3_4) to[out=-45,in= 165] (4_4);

\draw (0) to[out=165,in= 15] (-1); \draw (0) to[out=-165,in= -15] (-1);
\draw(-1)--(-2_1)--(-2_2)--(-1);
\draw (-2_1) to[out=135,in= -15] (-3_1); \draw (-2_1) to[out=165,in=-45] (-3_1);
\draw (-2_2) to[out=195,in=45] (-3_2); \draw (-2_2) to[out=225,in= 15] (-3_2);
\draw(-3_1)--(-4_1)--(-4_2)--(-3_1);
\draw(-3_2)--(-4_3)--(-4_4)--(-3_2);
\draw (-4_1) to[out=135,in= -15] (-5_1); \draw (-4_1) to[out=165,in=-45] (-5_1);
\draw (-4_2) to[out=195,in=45] (-5_2); \draw (-4_2) to[out=225,in= 15] (-5_2);
\draw (-4_3) to[out=135,in= -15] (-5_3); \draw (-4_3) to[out=165,in=-45] (-5_3);
\draw (-4_4) to[out=195,in=45] (-5_4); \draw (-4_4) to[out=225,in= 15] (-5_4);

\node at (-1,0.2) {\tiny$a$};
\node at (0,0.2) {\tiny$e$};
\node at (1,0.7) {\tiny$b$};
\node at (1,-0.7) {\tiny$b^2$};
\node at (2,1.4) {\tiny$ba$};
\node at (2,-1.4) {\tiny$b^2a$};
\node at (2.9,1.9) {\tiny$bab$};
\node at (2.9,0.5) {\tiny$bab^2$};
\node at (2.8,-0.5) {\tiny$b^2ab$};
\node at (2.8,-1.9) {\tiny$b^2ab^2$};

\node at (-2,0.7) {\tiny$ab$};
\node at (-2,-0.7) {\tiny$ab^2$};
\node at (-3,1.4) {\tiny$aba$};
\node at (-3,-1.4) {\tiny$ab^2a$};

\node at (-3.8,1.9) {\tiny$abab$};
\node at (-3.8,0.5) {\tiny$abab^2$};
\node at (-3.7,-0.5) {\tiny$ab^2ab$};
\node at (-3.8,-1.9) {\tiny$ab^2ab^2$};

\end{tikzpicture}
\caption{The Cayley graph of $\operatorname{PSL}_2\Z \simeq \langle a,b\mid a^2=1,  b^3=1\rangle$.} \label{fig:cayley}
\end{figure}
Our definition of $\Gamma$ is coherent with Serre's definition in \cite{MR1954121} making it a quotient of the Cayley graph of the rank 2 free group (the $4$-regular tree). Specifically, we associate double edges to the generating element $a$ of order 2. 

We explicitly solve the spectral problem for a projected image of the Laplacian on a line.
This allows us to provide an integral formula for the heat kernel of $\Gamma$ by elaborating on Chung–Yau’s method and extending their results on the relation between the heat kernels of two graphs related through a strong and regular covering.

One potential application of our result is the study of expander graphs. There has been an intense interest in finite quotients of $\operatorname{PSL}_2\Z$ and corresponding Cayley graphs with a fixed set of generators, see for example \cite{MR2569682, MR3144176, MR3348442}, or the seminal paper by Bourgain--Gamburd \cite{MR2415383}. 

The value
$$
\lambda_0:=\frac78-\frac12\sqrt{\frac{25}{16}+\sqrt{2}}=0.01234\dots
$$
is the bottom of the spectrum of the Laplacian on $\Gamma$ just as $1/4$ is for the Laplacian on $\mathbb{H}^2$. This can be compared with the work of Selberg \cite{MR182610}, where he conjectured that the smallest non-zero eigenvalue of the Laplacian on hyperbolic surfaces given as quotients of the upper half-plane $\mathbb{H}^2$ by congruence subgroups of $\operatorname{SL}_2\Z$ is bounded below by $1/4$. 
He also proved the lower bound $3/16$. Our numerical computations indicate that $\lambda_0$ is a lower bound for the first non-zero eigenvalue of the Laplacian of the Cayley graphs of $\operatorname{PSL}_2\mathbb{F}_p$ (with the same set of generators) for many primes $p$. In Fig.~\ref{fig:specL}, we plot the spectra for the first four primes. Further computations indicate that the first prime for which the first non-zero eigenvalue falls below $\lambda_0$ is $p=37$. 
We conjecture that $\lambda_0$ is the limiting value of the spectral gaps for the family $\operatorname{PSL}_2\mathbb{F}_p$.

Magee informed us that he made more general conjectures of similar type in his lecture \cite{magee_optimal_spectral_gaps_2024} in 2024, after which Sarnak pointed out that related conjectures to ours date back to a paper by Buck \cite{MR830648}. For further information, we also refer to the recent paper \cite{MR3985838}. 

\begin{figure}[h]
\center
\begin{tikzpicture}[xscale=7.5,line/.style = {draw,ultra thick, shorten >=-2pt, shorten <=-2pt}]
\draw[-latex] (-0.1,0) -- (1.9,0);

\node at (1.895,-0.25) {\small$\R$};

\node at (0,-0.4) {\small$0$}; \draw (0,0.15) -- (0,-0.15);

\node at (0.75,-0.4) {$\frac{3}{4}$};
\node at (1.75,-0.4) {$\frac{7}{4}$};

\draw[-latex] (0.02,0.5) -- (0.02,0.1);
\node at (0.02,0.7) {\small$\lambda_0$};

\draw[-latex] (0.68,0.5) -- (0.68,0.1);
\node at (0.68,0.7) {\small$\frac74-\lambda_1$};

\draw[-latex] (1.0675,0.5) -- (1.0675,0.1);
\node at (1.0675,0.7) {\small$\lambda_1$};

\draw[-latex] (1.73,0.5) -- (1.73,0.1);
\node at (1.73,0.7) {\small$\frac74-\lambda_0$};

\draw[line, {Circle[length=4pt]}-{Circle[length=4pt]}]  (0.75,0) -- (0.75,0);
\draw[line, {Circle[length=4pt]}-{Circle[length=4pt]}]  (1.75,0) -- (1.75,0);

\draw[line, {Circle[length=4pt]}-{Circle[length=4pt]}]  (0.02,0) -- (0.68,0);
\draw[line, {Circle[length=4pt]}-{Circle[length=4pt]}]  (1.0675,0) -- (1.73,0);

\draw[-latex] (-0.1,-1.55) -- (1.9,-1.55);

\node at (1.895,-1.8) {\small$\R$};

\draw (0,-1.35) -- (0,-1.7);
\draw[dashed] (0,-0.6) -- (0,-1.4); 
\draw[dashed] (0.75,-0.75) -- (0.75,-1.3); 
\draw[dashed] (1.75,-0.75) -- (1.75,-1.3); 
\draw[dashed] (0.02,-0.15) -- (0.02,-1.35); \draw (0.02,-1.45) -- (0.02,-1.65);
\draw[dashed] (0.68,-0.15) -- (0.68,-1.35); \draw (0.68,-1.45) -- (0.68,-1.65);
\draw[dashed] (1.0675,-0.15) -- (1.0675,-1.35); \draw (1.0675,-1.45) -- (1.0675,-1.65);
\draw[dashed] (1.73,-0.15) -- (1.73,-1.35); \draw (1.73,-1.45) -- (1.73,-1.65);

\draw[green!60!gray, line, {Circle[length=3.5pt]}-{Circle[length=3.5pt]}] (0,-1.325) -- (0,-1.325);
\draw[green!60!gray, line, {Circle[length=3.5pt]}-{Circle[length=3.5pt]}] (0.75,-1.325) -- (0.75,-1.325);
\draw[green!60!gray, line, {Circle[length=3.5pt]}-{Circle[length=3.5pt]}] (1,-1.55) -- (1,-1.55);
\draw[green!60!gray, line, {Circle[length=3.5pt]}-{Circle[length=3.5pt]}] (1.75,-1.325) -- (1.75,-1.325);

\draw[orange, line, {Circle[length=3.5pt]}-{Circle[length=3.5pt]}] (0,-1.47) -- (0,-1.47);
\draw[orange, line, {Circle[length=3.5pt]}-{Circle[length=3.5pt]}] (0.36,-1.55) -- (0.36,-1.55);
\draw[orange, line, {Circle[length=3.5pt]}-{Circle[length=3.5pt]}] (0.75,-1.47) -- (0.75,-1.47);
\draw[orange, line, {Circle[length=3.5pt]}-{Circle[length=3.5pt]}] (1.39,-1.55) -- (1.39,-1.55);
\draw[orange, line, {Circle[length=3.5pt]}-{Circle[length=3.5pt]}] (1.75,-1.47) -- (1.75,-1.47);

\draw[cyan, line, {Circle[length=3.5pt]}-{Circle[length=3.5pt]}] (0,-1.63) -- (0,-1.63);
\draw[cyan, line, {Circle[length=3.5pt]}-{Circle[length=3.5pt]}] (0.056,-1.55) -- (0.056,-1.55);
\draw[cyan, line, {Circle[length=3.5pt]}-{Circle[length=3.5pt]}] (0.075,-1.55) -- (0.075,-1.55);
\draw[cyan, line, {Circle[length=3.5pt]}-{Circle[length=3.5pt]}] (0.157,-1.55) -- (0.157,-1.55);
\draw[cyan, line, {Circle[length=3.5pt]}-{Circle[length=3.5pt]}] (0.165,-1.55) -- (0.165,-1.55);
\draw[cyan, line, {Circle[length=3.5pt]}-{Circle[length=3.5pt]}] (0.25,-1.55) -- (0.25,-1.55);
\draw[cyan, line, {Circle[length=3.5pt]}-{Circle[length=3.5pt]}] (0.5,-1.55) -- (0.5,-1.55);
\draw[cyan, line, {Circle[length=3.5pt]}-{Circle[length=3.5pt]}] (0.538,-1.55) -- (0.538,-1.55);
\draw[cyan, line, {Circle[length=3.5pt]}-{Circle[length=3.5pt]}] (0.542,-1.55) -- (0.542,-1.55);
\draw[cyan, line, {Circle[length=3.5pt]}-{Circle[length=3.5pt]}] (0.73,-1.55) -- (0.73,-1.55);
\draw[cyan, line, {Circle[length=3.5pt]}-{Circle[length=3.5pt]}] (0.75,-1.63) -- (0.75,-1.63);
\draw[cyan, line, {Circle[length=3.5pt]}-{Circle[length=3.5pt]}] (1.21,-1.55) -- (1.21,-1.55);
\draw[cyan, line, {Circle[length=3.5pt]}-{Circle[length=3.5pt]}] (1.25,-1.55) -- (1.25,-1.55);
\draw[cyan, line, {Circle[length=3.5pt]}-{Circle[length=3.5pt]}] (1.26,-1.55) -- (1.26,-1.55);
\draw[cyan, line, {Circle[length=3.5pt]}-{Circle[length=3.5pt]}] (1.49,-1.55) -- (1.49,-1.55);
\draw[cyan, line, {Circle[length=3.5pt]}-{Circle[length=3.5pt]}] (1.5,-1.55) -- (1.5,-1.55);
\draw[cyan, line, {Circle[length=3.5pt]}-{Circle[length=3.5pt]}] (1.58,-1.55) -- (1.58,-1.55);
\draw[cyan, line, {Circle[length=3.5pt]}-{Circle[length=3.5pt]}] (1.59,-1.55) -- (1.59,-1.55);
\draw[cyan, line, {Circle[length=3.5pt]}-{Circle[length=3.5pt]}] (1.69,-1.55) -- (1.69,-1.55);
\draw[cyan, line, {Circle[length=3.5pt]}-{Circle[length=3.5pt]}] (1.75,-1.63) -- (1.75,-1.63);

\draw[red, line, {Circle[length=3.5pt]}-{Circle[length=3.5pt]}] (0,-1.775) -- (0,-1.775);
\draw[red, line, {Circle[length=3.5pt]}-{Circle[length=3.5pt]}] (0.043,-1.55) -- (0.043,-1.55);
\draw[red, line, {Circle[length=3.5pt]}-{Circle[length=3.5pt]}] (0.11,-1.55) -- (0.11,-1.55);
\draw[red, line, {Circle[length=3.5pt]}-{Circle[length=3.5pt]}] (0.12,-1.55) -- (0.12,-1.55);
\draw[red, line, {Circle[length=3.5pt]}-{Circle[length=3.5pt]}] (0.25,-1.55) -- (0.25,-1.55);
\draw[red, line, {Circle[length=3.5pt]}-{Circle[length=3.5pt]}] (0.29,-1.55) -- (0.29,-1.55);
\draw[red, line, {Circle[length=3.5pt]}-{Circle[length=3.5pt]}] (0.41,-1.55) -- (0.41,-1.55);
\draw[red, line, {Circle[length=3.5pt]}-{Circle[length=3.5pt]}] (0.5,-1.55) -- (0.5,-1.55);
\draw[red, line, {Circle[length=3.5pt]}-{Circle[length=3.5pt]}] (0.75,-1.775) -- (0.75,-1.775);
\draw[red, line, {Circle[length=3.5pt]}-{Circle[length=3.5pt]}] (1.14,-1.55) -- (1.14,-1.55);
\draw[red, line, {Circle[length=3.5pt]}-{Circle[length=3.5pt]}] (1.25,-1.55) -- (1.25,-1.55);
\draw[red, line, {Circle[length=3.5pt]}-{Circle[length=3.5pt]}] (1.34,-1.55) -- (1.34,-1.55);
\draw[red, line, {Circle[length=3.5pt]}-{Circle[length=3.5pt]}] (1.46,-1.55) -- (1.46,-1.55);
\draw[red, line, {Circle[length=3.5pt]}-{Circle[length=3.5pt]}] (1.5,-1.55) -- (1.5,-1.55);
\draw[red, line, {Circle[length=3.5pt]}-{Circle[length=3.5pt]}] (1.63,-1.55) -- (1.63,-1.55);
\draw[red, line, {Circle[length=3.5pt]}-{Circle[length=3.5pt]}] (1.64,-1.55) -- (1.64,-1.55);
\draw[red, line, {Circle[length=3.5pt]}-{Circle[length=3.5pt]}] (1.71,-1.55) -- (1.71,-1.55);
\draw[red, line, {Circle[length=3.5pt]}-{Circle[length=3.5pt]}] (1.75,-1.775) -- (1.75,-1.775);

\node at (0.09,-2.5) {\small $p=2$,};
\draw[green!60!gray, line, {Circle[length=3.5pt]}-{Circle[length=3.5pt]}] (0,-2.5) -- (0,-2.5);
\node at (0.3,-2.5) {\small $p=3$,};
\draw[orange, line, {Circle[length=3.5pt]}-{Circle[length=3.5pt]}] (0.21,-2.5) -- (0.21,-2.5);
\node at (0.495,-2.5) {\small $p=5$,};
\draw[cyan, line, {Circle[length=3.5pt]}-{Circle[length=3.5pt]}] (0.41,-2.5) -- (0.41,-2.5);
\node at (0.695,-2.5) {\small $p=7$};
\draw[red, line, {Circle[length=3.5pt]}-{Circle[length=3.5pt]}] (0.61,-2.5) -- (0.61,-2.5);
\end{tikzpicture}
\caption{The spectra of Laplacians of $\operatorname{PSL}_2\Z$ (drawn in black) and $\operatorname{PSL}_2\mathbb{F}_p$, for $p=2,3,5,7$, where $\lambda_1=\frac78+\frac12\sqrt{\frac{25}{16}-\sqrt{2}}=1.0675\dots$.}  \label{fig:specL} 
\end{figure}

Kowalski in~\cite{MR3144176} proves explicit very small bounds for the spectral gap of families of Cayley graphs of $\operatorname{SL}_2 \mathbb{F}_p$. See \cite{MR4418463} for more recent results.

There is also a direct link between Selberg's conjecture and spectra of Cayley graphs of $\operatorname{SL}_2 \mathbb{F}_p$, see Helfgott \cite{MR3348442} section 5.5 for a discussion and references for this connection.

We now describe our main result.

Let $\mathcal{L}$ be the (normalized) Laplacian on $\Gamma$, which is a self-adjoint bounded linear operator in the Hilbert space $\ell^2(G)$ defined by

\begin{equation}\label{Laplacian_G}
(\mathcal{L} f)(x)=f(x)-\frac{1}{2}f(xa)-\frac{1}{4}f(xb)-\frac{1}{4}f(xb^{-1}),\quad \forall x\in G.
\end{equation}
The heat operator of $\Gamma$ is the exponential of the Laplacian, $h_t=e^{-t\mathcal{L}}$. The invariance of $\Gamma$ under left translations implies that
the heat kernel is described in terms of a function $k_t\colon G\to \R$ through the formula
\begin{equation}\label{heat_kernel_det}
(h_tf)(x)=\sum_{y\in G}k_t(y^{-1}x)f(y),\quad f\in \ell^2(G).
\end{equation}
Denote by $|x|$  the shortest word length of $x\in G$, and define the map
$\pi\colon G\to\Z$
by 
\begin{equation}\label{CY_proj}
 \pi(x)=
\begin{cases} 
0 & \text{if $x=e$} \\
 |x| & \text{if the reduced form of $x$ starts with letters $b^{\pm1}$} \\
 -|x| & \text{if the reduced form of $x$ starts with letter $a$}.
\end{cases}
\end{equation}

\begin{theorem} \label{K1}
The function $k_t$ determining the heat kernel in~\eqref{heat_kernel_det} is given by
${k}_t(x)=K_t(\pi(x))$, where, for $n\in\Z$,
\begin{equation} \label{eqK1}
 K_t(n)= e^{-t\frac{3}{4}} \alpha_n +e^{-t\frac{7}{4}} \beta_n + \int_0^\pi e^{-t(\frac{7}{8}-\frac{R_s}{2})}\gamma_n^-(s) \mathrm{d}s + \int_0^\pi e^{-t(\frac{7}{8}+\frac{R_s}{2})}\gamma_n^+(s) \mathrm{d}s,
\end{equation}
where the coefficients $\alpha_n,\beta_n$ and $\gamma^{\pm}_n(s)$ are defined depending on the sign and the parity of $n$. In the formulas with double lines below, the first line corresponds to even $n=2m$, and the second line to odd $n=2m+1$.

For $n\ge0$,
$$
\alpha_n= \frac{(-1)^{\left\lceil\frac{n}{2}\right\rceil}2^{- \left\lceil\frac{n}{2}\right\rceil }}{6},
\quad
\beta_n=\frac{(-1)^n 2^{- \left\lceil\frac{n}{2}\right\rceil }}{6},
$$
$$
\gamma_n^{\pm}(s) = \frac{\sqrt{2}^{-\left\lceil\frac{n}{2}\right\rceil} \sin(s)}{\pi R_s(1+8\sin^2(s))}
\begin{cases}
\mp(\sqrt{2}+\cos(s)) \sin(ms) + 4R_s\sin(s) \cos(ms) \\ 
\pm(4+2\sqrt{2}\cos(s)) \sin(ms) + \big(\sqrt{2}R_s\mp(\frac{9\sqrt{2}}{4}+4\cos(s))\big) \sin((m+1)s), 
\end{cases}
$$
and, for $n<0$,
$$
\alpha_n= \frac{(-1)^{\left\lceil\frac{n}{2}\right\rceil}}{6},
\quad
\beta_n=\frac{(-1)^n }{6},
$$
$$
\gamma_n^{\pm}(s)=\frac{\sqrt{2}^{-\left\lceil\frac{n}{2}\right\rceil} \sin(s)}{\pi R_s (1+8\sin^2(s))} 
\begin{cases}
\pm(\sqrt{2}+\cos(s)) \sin(ms) + 4R_s\sin(s)\cos(ms) \\
\pm(4+2\sqrt{2}\cos(s)) \sin(ms) - \big(\sqrt{2}R_s\pm(\frac{9\sqrt{2}}{4}+4\cos(s))\big) \sin((m+1)s), 
\end{cases} 
$$
where 
\begin{equation}\label{R_s}
R_s=\sqrt{\frac{25}{16}+\sqrt{2}\cos(s)}.
\end{equation}
\end{theorem}
Note that expression \eqref{eqK1} simplifies greatly for $n=0$.
Using the limiting values $s=0$ and $s=\pi$ in the formula for $R_s$, we can describe the spectrum of the Laplacian on $\Gamma$. 
\begin{corollary}\label{Cor_spec}
 The spectrum of the Laplacian $\mathcal{L}$ on $\ell^2(G)$ is the following closed subset of $\R$, see Fig.~\ref{fig:specL}:
$$
\operatorname{Sp}(\mathcal{L})=\Big[\lambda_0,\frac74-\lambda_1\Big]\sqcup\Big\{\frac{3}{4}\Big\}\sqcup\Big[\lambda_1,\frac74-\lambda_0\Big]\sqcup\Big\{\frac{7}{4}\Big\} 
$$
where
$$
\lambda_0=\frac78-\frac12\sqrt{\frac{25}{16}+\sqrt{2}}=0.01234\dots,\quad \lambda_1=\frac78+\frac12\sqrt{\frac{25}{16}-\sqrt{2}}=1.0675\dots.
$$
\end{corollary}

Note that the isolated points are indeed eigenvalues. Proposition~\ref{spectrum_L^pr} describes the eigenfunctions of the projected Laplacian, and Proposition~\ref{efct_cov} shows how to lift them to eigenfunctions of the Laplacian on $\ell^2(G)$.

Corollary~\ref{Cor_spec} should be compared with Theorem 1 from Cartwright--Soardi in \cite{MR820357}, where they treat the general free product of two cyclic groups. However, note that our Laplacian is somewhat different, as we consider double edges for the generator $a=a^{-1}$. The nature of the methods is very different; unlike their work, our technique is through spectral problem resolution. Magee pointed out to us that the use of the $R$-transform and free probability provides another approach to determining the spectrum.

{\bf Outline.} 
Section~\ref{sec2} contains basics of weighted graphs: their definition and coverings.
In Section~\ref{sec3}, we discuss coverings of weighted graphs and provide geometric insight into this notion by defining quotient weighted graphs on the basis of groups of automorphisms of graphs. Then, we prove Proposition~\ref{Prop2}, the main contribution of this section, that reduces the verification of a map to be a covering to a group theoretical problem. 
In Section~4, we describe Chung--Yau's result on the relation between heat kernels of two weighted graphs if one covers another strongly and regularly.
Finally, in Section~5, we solve the spectral problem for the projected Laplacian and use the resulting spectral decomposition to prove the explicit heat kernel formula, completing the proof of Theorem~\ref{K1}.

\section{Basics of weighted graphs}\label{sec2}
We start by defining the basics of weighted graphs.
\begin{definition}
 A \textit{weighted graph} is a set of \textit{vertices} $V$ provided with a non-negative symmetric \textit{weight function} 
 $$
 w\colon V\times V\to \R_{\geq 0},\quad w(u,v)=w(v,u)\quad \forall u,v \in V.
 $$
 By abuse of notation, sometimes we will not distinguish between a weighted graph $(V,w)$ and its underlying set of vertices $V$ when the weight function $w$ is clear from the context, so that we simply write $V$ instead of $(V,w)$.

 In a weighted graph, an \textit{edge} is a pair of vertices with strictly positive weight, and two vertices connected by an edge are called \textit{adjacent}.
 
 The \textit{degree} of a vertex $u\in V$ is defined as 
 $$
 d_u=\sum_{v\in V}w(u,v).
 $$
 
 We say that a graph $V$ is \textit{$k$-regular} if $d_u=k$, $\forall u\in V$.
\end{definition}

\begin{remark}\label{top_graph}
A topological graph is a weighted graph, where the weight $w(x,y)$ is defined as the number of edges between $x$ and $y$, where, in the case $x=y$, each edge is counted twice. 
\end{remark}

\begin{remark}
 Any metric space is a weighted graph, where the weight function is the distance, but the degrees can be infinite. In this work, all degrees are assumed to be finite.
\end{remark}

\begin{definition}
Given two weighted graphs $(\tilde{V},\tilde{w})$, $(V,w)$, a map $\pi \colon \tilde{V} \to V$ is called \textit{covering of weighted graphs} if $\pi(\tilde{V})=V$ and if there exists a function $\lambda: V \to \R_{>0}$ such that for any $u\in V$ and any $y\in \tilde{V}$, the following equation holds
 \begin{equation}\label{new_def}
\sum_{x\in \pi^{-1}(u)}\tilde{w}(x,y)=\lambda(\pi(y)) w(u,\pi(y)).
\end{equation}
We say that $(\tilde{V},\tilde{w})$ \textit{covers} $(V,w)$ through $\pi$. 
\end{definition} 

\begin{remark}
 This definition is a modification of the one given by Chung--Yau~\cite{MR1667452}, made to allow infinite fibers. 
\end{remark}
 
\begin{remark}
 In this work, unless otherwise specified, we assume that all graphs are connected, which means that for all $u,v\in V$, there exists a sequence of edges connecting $u$ and $v$. Consequently, this means that for all $u\in V$, $d_u\neq0$.
\end{remark}

\begin{remark}\label{new_def_finite_fibers}
For a covering of (connected) weighted graphs $\pi \colon \tilde{V} \to V$, if $\pi^{-1}(u)$ is a finite set for any $u\in V$, then $\lambda(u)=\frac{c}{|\pi^{-1}(u)|}$ for some constant $c>0$.

 Indeed, let $u,v \in V$ be such that $w(u,v)\ne0$ and $y\in\pi^{-1}(v)$. Since $\pi$ is a covering, we have 
 $$
 \sum_{x\in \pi^{-1}(u)}\tilde{w}(x,y)=\lambda(v) w(u,v).
 $$
  Summing over $y\in \pi^{-1}(v)$ both sides of this equality, we obtain
 $$
 \sum_{y\in \pi^{-1}(v)} \sum_{x\in \pi^{-1}(u)}\tilde{w}(x,y) =\vert \pi^{-1}(v)\vert \lambda(v)w(u,v).
 $$
 Using the fact that the left hand side is symmetric in $u$ and $v$, and the condition  $w(u,v)\ne0$, we conclude that 
$$
\vert \pi^{-1}(u)\vert \lambda(u)=\vert \pi^{-1}(v)\vert \lambda(v),
$$
concluding that the quantity $\vert \pi^{-1}(u)\vert \lambda(u)$ is constant for all $u\in V$. 
\end{remark}

\section{Coverings of weighted graphs}\label{sec3}
In this section, we rework some of the results from~\cite{MR1667452} and present a general group theoretical approach to quotient graphs.
\subsection{Properties}

Any topological covering is a covering of weighted graphs with $\lambda=1$. However, a covering of weighted graphs is not necessarily a topological covering since the preimages (fibers) of vertices can have different cardinalities.

We also define a more specific class of coverings, which will be useful for us.

\begin{definition}
Given two weighted graphs $(\tilde{V},\tilde{w})$ and $(V,w)$, we say $(\tilde{V},\tilde{w})$ \textit{covers} $(V,w)$ \textit{strongly} and \textit{regularly} if there exists a vertex $u_0\in V$, called \textit{distinguished vertex}, such that, for any vertex $x\in \tilde{V}$ there exists a covering of weighted graphs $\pi_x \colon (\tilde{V},\tilde{w})\to (V,w)$ such that $\pi_x^{-1}(u_0)=\{x\}$.
\end{definition}

\begin{remark}
In the case of topological graphs, $\tilde{\Gamma}$ covers $\Gamma$ strongly and regularly in the topological sense if and only if they are isomorphic. 
\end{remark}

We note that the degrees of the vertices of a graph $V$ and a covering graph $\tilde V$ are not necessarily equal. The following lemma shows how they are related.

\begin{lemma}
 Let $(\tilde{V},\tilde{w})$ and $(V,w)$ be two weighted graphs such that $(\tilde{V},\tilde{w})$ covers $(V,w)$ through $\pi$.
 Then, the following relation between degrees in $\tilde{V}$ and $V$ holds
 \begin{equation}\label{new_degree}
 \tilde{d}_x=\lambda(\pi(x))d_{\pi(x)}, \quad\forall x\in \tilde{V}.
 \end{equation}
 In particular, $\tilde{d}_x=\tilde{d}_y$ if $\pi(x)=\pi(y)$.
\end{lemma}

\begin{proof}
Let $x\in\tilde{V}$. Using the definition of degree of a vertex and equation~\eqref{new_def} we have 
$$
\tilde{d}_x  =\sum_{y\in\tilde{V}}\tilde{w}(x,y) = \sum_{u\in V}\sum_{y\in\pi^{-1}(u)}\tilde{w}(x,y) = \sum_{u\in V}\lambda(\pi(x))w(\pi(x),u) = \lambda(\pi(x))d_{\pi(x)}.
$$
\end{proof}

\subsection{Quotient weighted graphs}
Here we establish a result which allows to reduce the verification of the property of a covering with finite fibers to a group theoretical problem, which can often facilitate verifications by using a geometrical argument.

Given a weighted graph $(V,w)$, we let $\operatorname{Aut}(V,w)$ denote the group of automorphisms of the weighted graph, that is, the set of permutations $\sigma$ of $V$ which preserve the weight function
$$
w(\sigma(u),\sigma(v))=w(u,v) \quad \forall u,v\in V.
$$

\begin{definition}
 Let $(\tilde{V},\tilde{w})$ be a weighted graph and $G\subset \operatorname{Aut}(\tilde{V},\tilde{w})$ a subgroup such that the orbit $Gx$ of $x$ is a finite set for any $x\in\tilde{V}$. The \textit{quotient weighted graph} of the weighted graph $(\tilde{V},\tilde{w})$ with respect to the group $G$ is a weighted graph $(V,w)$  defined as
 
\begin{equation} \label{qgraph}
 V=\tilde{V}/G, \quad w(u,v)=\sum_{\substack{x\in u\\ y\in v}}\tilde{w}(x,y) \quad \forall u,v\in V.
\end{equation}
\end{definition}
 
 Before stating our result, we recall the following group theoretical fact, which will be needed in the proof.
\begin{lemma}[Orbit-stabilizer theorem]
Let a group $G$ act on a set $X$. Then, for any $x\in X$, the orbit $Gx=\{gx\mid g\in G\}$ is in bijection with the set of cosets for the stabilizer subgroup $G/H_x:=\{gH_x\mid g\in G\}$, where $H_x:=\{h\in G\mid hx=x\}\subset G$ is the stabilizer subgroup of $x$.
\end{lemma}

\begin{proposition}
\label{Prop2}
 Let $(\tilde{V},\tilde{w})$ be a weighted graph, $G\subset \operatorname{Aut}(\tilde{V},\tilde{w})$ a subgroup such that the orbit $Gx$ of $x$ is a finite set for any $x\in\tilde{V}$ and $(V,w)$ the \textit{quotient weighted graph} of the weighted graph $(\tilde{V},\tilde{w})$ with respect to the group $G$. Then, the canonical projection map to the quotient space $\pi\colon\tilde{V}\to V$ is a covering of weighted graphs.
\end{proposition}

\begin{proof}
 Let $x\in\tilde{V}$, $H_x:=\{h\in G\mid hx=x\}\subset G$ the \textit{stabilizer} subgroup of $x$, $v=\pi(x)$, $u\in V$ and $s\colon G/H_x\to G$ a map such that $s(\alpha)H_x=\alpha$, $\forall\alpha\in G/H_x$ (this means that a representative $s(\alpha)$ is chosen in each coset $\alpha=gH_x$).
 Then,
\begin{align*}
    w(u,v) & =\sum_{\substack{z\in \pi^{-1}(u)\\ x'\in \pi^{-1}(v)}}\tilde{w}(z,x') & \text{by def. of $w$}  \\
    & =\sum_{\substack{z\in \pi^{-1}(u) \\ \alpha\in G/H_x}}\tilde{w}(z,s(\alpha)x) & \text{using the bijection $G/H_x\to Gx=\pi^{-1}(v)$} \\
    & = \sum_{\substack{z\in \pi^{-1}(u) \\ \alpha\in G/H_x}} \tilde{w}({s(\alpha)}^{-1}z,x) & \text{since $s(\alpha)\in \operatorname{Aut}(\tilde{V},\tilde{w})$ } \\
    & =\sum_{\alpha\in G/H_x}\sum_{z'\in \pi^{-1}(u)}\tilde{w}(z',x) & \text{substituting $s(\alpha)^{-1}z$ by $z'$} \\ 
    & =\vert G/H_x\vert \sum_{z'\in \pi^{-1}(u)}\tilde{w}(z',x) & \text{since the internal sum is independent of $\alpha$} \\
    & =\vert \pi^{-1}(v)\vert\sum_{z\in \pi^{-1}(u)}\tilde{w}(z,x) & \text{using $\vert G/H_x\vert=\vert \pi^{-1}(u) \vert$ and substituting $z'$ by $z$}\\ 
\end{align*}
 which is exactly formula~\eqref{new_def} with $\lambda(v)=\frac{1}{|\pi^{-1}(v)|}$.
\end{proof}

\section{Spectrum of the Laplacian and coverings}
\subsection{The Laplacian matrix}
In this section, we describe a result from~\cite{MR1667452} that relates the heat kernels of a graph and a covering graph.

The \emph{matrix coefficients} $A(u,v)$ of a linear map (operator) $A\colon \C^V \to \C^V$ are defined by 
$$
(Af)(u)=\sum_{v\in V}A(u,v)f(v).
$$

\begin{definition} 
 Given a weighted graph $(V,w)$ the \textit{normalized Laplacian} $\mathcal{L}$ of $(V,w)$ is an operator with the matrix coefficients  
\begin{equation}\label{normLaplacian}
  \mathcal{L}(u,v)=\delta_{u,v}-\frac{w(u,v)}{\sqrt{d_ud_v}}.
\end{equation}
 \end{definition}
 
This is a self-adjoint (bounded) operator on the Hilbert space $\ell^2(V)$.

The following lemma describes the relation between the Laplacians on a graph and on a covering graph.
\begin{lemma}\label{norm_lap_cov}
 Let $(\tilde{V},\tilde{w})$ cover $(V,w)$ through $\pi$ with associated function $\lambda: V \to \R_{>0}$. Let $\tilde{\mathcal{L}}$ and $\mathcal{L}$ denote the corresponding (normalized) Laplacians.
 Then, for any $u\in V$ and any $y\in \tilde{V}$, one has the equality
 \begin{equation}\label{norm_lap_cov_rel}
\sum_{x\in\pi^{-1}(u)}\tilde{\mathcal{L}}(x,y)=\sqrt{\frac{\lambda(\pi(y))}{\lambda(u)}} \mathcal{L} (u,\pi(y)).
\end{equation}
\end{lemma}

\begin{proof}
It follows from equations~\eqref{new_def} and~\eqref{new_degree}.
Indeed, let $u\in V$ and $y\in \tilde{V}$. We have
\begin{align*}
 \sum_{x\in\pi^{-1}(u)} \tilde{\mathcal{L}}(x,y) & = \sum_{x\in\pi^{-1}(u)} \Big( \delta_{x,y} - \frac{\tilde w(x,y)}{\sqrt{\tilde d_x\tilde d_y}} \Big) & \text{by def. of $\tilde{\mathcal{L}}$}\\
 & = \delta_{\pi(y),u} - \frac{1}{\sqrt{\lambda(u)d_u\lambda(\pi(y))d_{\pi(y)}}}\sum_{x\in\pi^{-1}(u)}\tilde w(x,y) & \text{by \eqref{new_degree}}\\
 & = \delta_{\pi(y),u} - \frac{1}{\sqrt{\lambda(u)d_u\lambda(\pi(y))d_{\pi(y)}}}\lambda(\pi(y))w(u,\pi(y)) & \text{by \eqref{new_def}}\\
 & = \delta_{\pi(y),u} - \sqrt{\frac{\lambda(\pi(y))}{\lambda(u)}}\frac{w(u,\pi(y))}{\sqrt{d_ud_{\pi(y)}}} \\
 & = \sqrt{\frac{\lambda(\pi(y))}{\lambda(u)}}\mathcal{L}(u,\pi(y)).
\end{align*}
\end{proof}

The next proposition describes how eigenfunctions of the normalized Laplacian lift through a covering.
\begin{proposition}\label{efct_cov}
 Let $(\tilde{V},\tilde{w})$ cover $(V,w)$ through $\pi$ with associated function $\lambda: V \to \R_{>0}$, and assume that all fibers $\pi^{-1}(v)$ are finite. Let $\tilde{\mathcal{L}}$ and $\mathcal{L}$ denote the corresponding (normalized) Laplacians. If $f\in \ell^2(V)$ is an eigenfunction of $\mathcal{L}$ with  eigenvalue $\mu$, then the pull-back $\pi^{*}\sqrt{\lambda}f \in \ell^2(\tilde{V})$ is an eigenfunction of $\tilde{\mathcal{L}}$ with the same eigenvalue $\mu$. Explicitly, 
 $$
 (\pi^{*}\sqrt{\lambda}f)(x)=\sqrt{\lambda(\pi(x))}f(\pi(x)),
 $$
 and 
 the function $\lambda$ is given by
$$
\lambda(v)=\frac{c}{|\pi^{-1}(v)|}
$$
 for a constant $c>0$, as described in Remark~\ref{new_def_finite_fibers}.
\end{proposition}

\begin{proof}
 Suppose $f\in \ell^2(V)$ is an eigenfunction of $\mathcal{L}$ with eigenvalue $\mu$, that is $\mathcal{L}f=\mu f$. Let $x\in \tilde{V}$. We have
 \begin{align*}
 (\tilde{\mathcal{L}}\pi^{*}\sqrt{\lambda}f)(x) & = \sum_{y\in\tilde{V}} \tilde{\mathcal{L}}(x,y) \sqrt{\lambda(\pi(y))}f(\pi(y)) 
  = \sum_{v\in V} \sum_{y\in\pi^{-1}(v)} \tilde{\mathcal{L}}(x,y) \sqrt{\lambda(v)}f(v) \\
 & = \sum_{v\in V} \sqrt{\lambda(v)}f(v) \sum_{y\in\pi^{-1}(v)} \tilde{\mathcal{L}}(x,y)  & \text{}\\
 & = \sum_{v\in V} \sqrt{\lambda(v)}f(v) \sqrt{\frac{\lambda(\pi(x))}{\lambda(v)}} \mathcal{L} (\pi(x),v)  & \text{by \eqref{norm_lap_cov_rel}}\\
 & = \sqrt{\lambda(\pi(x))} \sum_{v\in V}  \mathcal{L} (\pi(x),v) f(v) =  \sqrt{\lambda(\pi(x))} (\mathcal{L}f)(\pi(x)) \\
 & = \sqrt{\lambda(\pi(x))} \mu f(\pi(x)) = \mu (\pi^{*}\sqrt{\lambda}f)(x) & \text{since $\mathcal{L}f=\mu f$}.
\end{align*}
Let us show that $\pi^{*}\sqrt{\lambda}f \in \ell^2(\tilde{V})$. We have
 \begin{align*}
 \| \pi^{*}\sqrt{\lambda}f \|^2 & = \sum_{x\in \tilde{V}} |(\pi^{*}\sqrt{\lambda}f)(x)|^2 
 = \sum_{x\in \tilde{V}} \lambda(\pi(x)) |f(\pi(x))|^2 \\
 & = \sum_{v\in V}\sum_{x\in \pi^{-1}(v)} \lambda(v) |f(v)|^2 = \sum_{v\in V} |\pi^{-1}(v)| \lambda(v) |f(v)|^2 \\
 & = c \sum_{v\in V} |f(v)|^2 = c \|f\|^2 < \infty & \text{since $\lambda(v)=\frac{c}{|\pi^{-1}(v)|}.$}\\
\end{align*}
This completes the proof.
\end{proof}

\subsection{The heat kernel}
We start by defining the heat operator of a graph. 

\begin{definition}\label{ht}
 Given a weighted graph $(V,w)$, the \textit{heat operator} $h_t$ of $(V,w)$ is an operator defined for $t\geq 0$ as 
 $$
 h_t=e^{-t\mathcal{L}}=\sum_{k=0}^{\infty}\frac{(-t)^k}{k!}\mathcal{L}^k.
 $$
\end{definition}
The basic problem is to determine the heat kernel $h_t(x,y)$.
\begin{remark}
For a Cayley graph $\Gamma(G)$ of a group $G$, the invariance of $\Gamma(G)$ under left translations implies that $h_t(x,y)=k_t(y^{-1}x)$, where $k_t(x):=h_t(x,e)$. In this case, the problem of determining the matrix coefficients of the heat operator is reduced to the problem of determining the function $k_t(x)$.

 The function $k_t\colon G\to\C$ 
is the unique solution of the following differential equation with initial condition
\begin{equation}\label{hk}
 \begin{cases} 
      \frac{\partial}{\partial{t}}k_t=-\mathcal{L}k_t \\
      k_0(x)=\delta_{e,x}.
\end{cases}
\end{equation}
\end{remark}

\begin{lemma}\label{norm_lapk_cov}
 Let $(\tilde{V},\tilde{w})$ cover $(V,w)$ through $\pi$ with associated function $\lambda: V \to \R_{>0}$. Let $\tilde{\mathcal{L}}$ and $\mathcal{L}$ denote the corresponding (normalized) Laplacians.
 Then, for any $k\in\Z_{\ge0}$, any $u,v\in V$ and any $y\in\pi^{-1}(v)$, one has the equality
 \begin{equation}\label{norm_lapk_cov_rel}
\sum_{x\in\pi^{-1}(u)}\tilde{\mathcal{L}}^k(x,y)=\sqrt{\frac{\lambda(v)}{\lambda(u)}}\mathcal{L}^k(u,v).
\end{equation}
\end{lemma}

\begin{proof}
By taking into account the fact that $\tilde{\mathcal{L}}^0$ and $\mathcal{L}^0$ are the identity operators, equality~\eqref{norm_lapk_cov_rel} with $k=0$ is verified as follows
$$
 \sum_{x\in\pi^{-1}(u)} \tilde{\mathcal{L}}^0(x,y) = \sum_{x\in\pi^{-1}(u)}\delta_{x,y}=\delta_{u,\pi(y)}=\delta_{u,v} =\sqrt{\frac{\lambda(v)}{\lambda(u)}}\delta_{u,v}=\sqrt{\frac{\lambda(v)}{\lambda(u)}} \mathcal{L}^0(u,v).
$$

The case $k=1$ is precisely Lemma~\ref{norm_lap_cov}.

Now, we proceed by induction. Assume that equality~\eqref{norm_lapk_cov_rel} is satisfied for all $k\in\{0,1,\dots,r\}$ for $r\ge1$. Then, we have
\begin{align*}
 \sum_{x\in\pi^{-1}(u)} \tilde{\mathcal{L}}^{r+1}(x,y) & = \sum_{x\in\pi^{-1}(u)}\sum_{z\in\tilde V} \tilde{\mathcal{L}}^{r}(x,z) \tilde{\mathcal{L}}(z,y) 
  = \sum_{z\in\tilde V}\Big(\sum_{x\in\pi^{-1}(u)} \tilde{\mathcal{L}}^{r}(x,z)\Big) \tilde{\mathcal{L}}(z,y)\\
 & = \sum_{z\in\tilde V}\sqrt{\frac{\lambda(\pi(z))}{\lambda(u)}} \mathcal{L}^r(u,\pi(z)) \tilde{\mathcal{L}}(z,y) \\
 & = \sum_{a\in V}\sqrt{\frac{\lambda(a)}{\lambda(u)}} \mathcal{L}^r(u,a)\sum_{z\in\pi^{-1}(a)} \tilde{\mathcal{L}}(z,y) \\
 & = \sum_{a\in V}\sqrt{\frac{\lambda(a)}{\lambda(u)}}\mathcal{L}^r(u,a)\sqrt{\frac{\lambda(v)}{\lambda(a)}}\mathcal{L}(a,v)\\
 & = \sqrt{\frac{\lambda(v)}{\lambda(u)}}\sum_{a\in V}\mathcal{L}^r(u,a)\mathcal{L}(a,v)
  = \sqrt{\frac{\lambda(v)}{\lambda(u)}}\mathcal{L}^{r+1}(u,v)
\end{align*}
where, 
in the third equality, we used the induction hypothesis for $k=r$,
and in the fifth equality, we used formula~\eqref{norm_lapk_cov_rel} for $k=1$.
\end{proof}

\begin{proposition} \label{rel_covering_heat_kernel}
Let $(\tilde{V},\tilde{w})$ cover $(V,w)$ through $\pi$ with associated function $\lambda: V \to \R_{>0}$. Let $\tilde{h}_t$ and $h_t$ denote the corresponding heat operators.
 Then, for any $u,v\in V$ and any $y\in\pi^{-1}(v)$, one has the equality
\begin{equation}\label{eq5}
\sum_{x\in\pi^{-1}(u)}\tilde{h}_t(x,y)=\sqrt{\frac{\lambda(v)}{\lambda(u)}}h_t(u,v).
\end{equation}
In particular, suppose that the covering is strong and regular. Choose a distinguished vertex $u_0\in V$ and a covering 
$\pi_x\colon\tilde{V}\to V$ such that $\pi_x^{-1}(u_0)=\{x\}$. Then, we have, for any $v\in V$ and $y\in\pi_x^{-1}(v)$,
 \begin{equation}\label{new_hk}
 \tilde{h}_t(x,y)=\sqrt{\frac{\lambda(v)}{\lambda(u_0)}}h_t(u_0,v)=\sqrt{\frac{\lambda(\pi_x(y))}{\lambda(\pi_x(x))}}h_t(\pi_x(x),\pi_x(y)),
 \end{equation}
 since $\pi_x(x)=u_0$.
\end{proposition}

\begin{proof}
Let $u,v\in V$ and $y\in\pi^{-1}(v)$. Then, using $h_t=\sum_{k=0}^{\infty}\frac{(-t)^k}{k!}\mathcal{L}^k$, we have 
\begin{align*}
 & \sum_{x\in\pi^{-1}(u)}\tilde{h}_t(x,y) 
 = \sum_{x\in\pi^{-1}(u)}\sum_{k=0}^{\infty}\frac{(-t)^k}{k!}\tilde{\mathcal{L}}^k(x,y)
 = \sum_{k=0}^{\infty}\frac{(-t)^k}{k!}\sum_{x\in\pi^{-1}(u)}\tilde{\mathcal{L}}^k(x,y) \\
 & \quad \quad \quad = \sum_{k=0}^{\infty}\frac{(-t)^k}{k!}\sqrt{\frac{\lambda(v)}{\lambda(u)}}\mathcal{L}^k(u,v) \, = \, \sqrt{\frac{\lambda(v)}{\lambda(u)}}\sum_{k=0}^{\infty}\frac{(-t)^k}{k!}\mathcal{L}^k(u,v) = \, \sqrt{\frac{\lambda(v)}{\lambda(u)}}h_t(u,v)
\end{align*}
where, in the third equality, we used Lemma~\ref{norm_lapk_cov}.
\end{proof}

\begin{remark}
 In the case when $\pi \colon \tilde{V} \to V$ is a strong and regular covering with finite fibers, using remark~\ref{new_def_finite_fibers}, 
 equality~\eqref{new_hk} takes the form 
\begin{equation}\label{eq6}
 \tilde{h}_t(x,y)=\frac{1}{\sqrt{\vert\pi_x^{-1}(\pi_x(y))\vert}}h_t(\pi_x(x),\pi_x(y)).
\end{equation}
\end{remark}

\section{The heat kernel of the Cayley graph $\Gamma(C_2*C_3)$}
In this section, we establish a formula for the heat kernel of the Cayley graph $\Gamma:=\Gamma(G)$ of the group $G=\operatorname{PSL}_2\Z\simeq C_2*C_3$ of the presentation $\langle a,b\mid a^2=1,  b^3=1\rangle$.
In identification with $\operatorname{PSL}_2\Z$, we can represent the generators by the matrices
$a=\left(\begin{smallmatrix} 0 & -1 \\ 1 & 0 \end{smallmatrix}\right)$ and $b=\left(\begin{smallmatrix} 0 & -1 \\ 1 & 1 \end{smallmatrix}\right)$.

\subsection{Construction of the covering}
According to Fig.~\ref{fig:covering}, we observe that $\Gamma$
covers a weighted line $L_{\infty}$ on the vertex set $\Z$, where the line is essentially determined through the distance function on the graph~\eqref{CY_proj}.

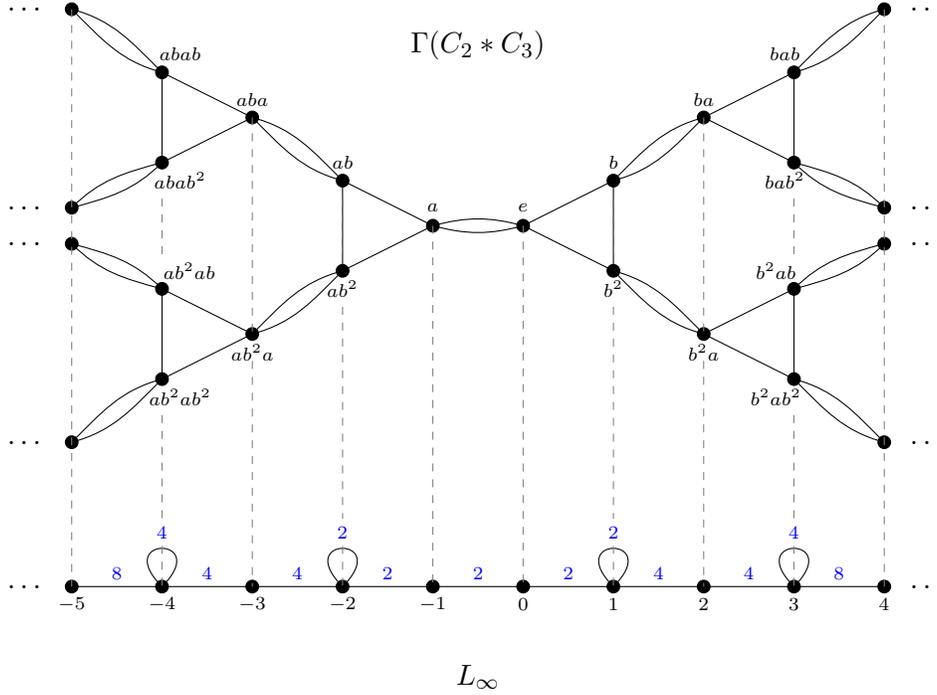
\begin{figure}
\center
\begin{tikzpicture}[scale=1.2]

\node at (-.5,2){$\Gamma(C_2*C_3)$};

\coordinate (0) at (0,0);
\coordinate (1_1) at (1,0.5);
\coordinate (1_2) at (1,-0.5);
\coordinate (2_1) at (2,1.2);
\coordinate (2_2) at (2,-1.2);

\begin{scope}[shift={(2,1.2)}]
\coordinate (3_1) at (1,0.5);
\coordinate (3_2) at (1,-0.5);
\coordinate (4_1) at (2,1.2);
\coordinate (4_2) at (2,-1);
\node at (2.5,1.2){$\dots$};
\node at (2.5,-1){$\dots$};
\end{scope}

\begin{scope}[shift={(2,-1.2)}]
\coordinate (3_3) at (1,0.5);
\coordinate (3_4) at (1,-0.5);
\coordinate (4_3) at (2,1);
\coordinate (4_4) at (2,-1.2);
\node at (2.5,1){$\dots$};
\node at (2.5,-1.2){$\dots$};
\end{scope}

\coordinate (-1) at (-1,0);
\begin{scope}[shift={(-1,0)}]
\coordinate (-2_1) at (-1,0.5);
\coordinate (-2_2) at (-1,-0.5);
\coordinate (-3_1) at (-2,1.2);
\coordinate (-3_2) at (-2,-1.2);
\end{scope}

\begin{scope}[shift={(-3,1.2)}]
\coordinate (-4_1) at (-1,0.5);
\coordinate (-4_2) at (-1,-0.5);
\coordinate (-5_1) at (-2,1.2);
\coordinate (-5_2) at (-2,-1);
\node at (-2.5,1.2){$\dots$};
\node at (-2.5,-1){$\dots$};
\end{scope}

\begin{scope}[shift={(-3,-1.2)}]
\coordinate (-4_3) at (-1,0.5);
\coordinate (-4_4) at (-1,-0.5);
\coordinate (-5_3) at (-2,1);
\coordinate (-5_4) at (-2,-1.2);
\node at (-2.5,1){$\dots$};
\node at (-2.5,-1.2){$\dots$};
\end{scope}

\draw[fill=black] (0) circle [radius=2pt];
\draw[fill=black] (1_1) circle [radius=2pt];
\draw[fill=black] (1_2) circle [radius=2pt];
\draw[fill=black] (2_1) circle [radius=2pt];
\draw[fill=black] (2_2) circle [radius=2pt];
\draw[fill=black] (3_1) circle [radius=2pt];
\draw[fill=black] (3_2) circle [radius=2pt];
\draw[fill=black] (3_3) circle [radius=2pt];
\draw[fill=black] (3_4) circle [radius=2pt];
\draw[fill=black] (4_1) circle [radius=2pt];
\draw[fill=black] (4_2) circle [radius=2pt];
\draw[fill=black] (4_3) circle [radius=2pt];
\draw[fill=black] (4_4) circle [radius=2pt];
\draw[fill=black] (-1) circle [radius=2pt];
\draw[fill=black] (-2_1) circle [radius=2pt];
\draw[fill=black] (-2_2) circle [radius=2pt];
\draw[fill=black] (-3_1) circle [radius=2pt];
\draw[fill=black] (-3_2) circle [radius=2pt];
\draw[fill=black] (-4_1) circle [radius=2pt];
\draw[fill=black] (-4_2) circle [radius=2pt];
\draw[fill=black] (-4_3) circle [radius=2pt];
\draw[fill=black] (-4_4) circle [radius=2pt];
\draw[fill=black] (-5_1) circle [radius=2pt];
\draw[fill=black] (-5_2) circle [radius=2pt];
\draw[fill=black] (-5_3) circle [radius=2pt];
\draw[fill=black] (-5_4) circle [radius=2pt];

\draw(0)--(1_1)--(1_2)--(0);
\draw (1_1) to[out=45,in= 195] (2_1); \draw (1_1) to[out=15,in= 225] (2_1);
\draw (1_2) to[out=-15,in=135] (2_2); \draw (1_2) to[out=-45,in= 165] (2_2);
\draw(2_1)--(3_1)--(3_2)--(2_1);
\draw(2_2)--(3_3)--(3_4)--(2_2);
\draw (3_1) to[out=45,in= 195] (4_1); \draw (3_1) to[out=15,in= 225] (4_1);
\draw (3_2) to[out=-15,in=135] (4_2); \draw (3_2) to[out=-45,in= 165] (4_2);
\draw (3_3) to[out=45,in= 195] (4_3); \draw (3_3) to[out=15,in= 225] (4_3);
\draw (3_4) to[out=-15,in=135] (4_4); \draw (3_4) to[out=-45,in= 165] (4_4);

\draw (0) to[out=165,in= 15] (-1); \draw (0) to[out=-165,in= -15] (-1);
\draw(-1)--(-2_1)--(-2_2)--(-1);
\draw (-2_1) to[out=135,in= -15] (-3_1); \draw (-2_1) to[out=165,in=-45] (-3_1);
\draw (-2_2) to[out=195,in=45] (-3_2); \draw (-2_2) to[out=225,in= 15] (-3_2);
\draw(-3_1)--(-4_1)--(-4_2)--(-3_1);
\draw(-3_2)--(-4_3)--(-4_4)--(-3_2);
\draw (-4_1) to[out=135,in= -15] (-5_1); \draw (-4_1) to[out=165,in=-45] (-5_1);
\draw (-4_2) to[out=195,in=45] (-5_2); \draw (-4_2) to[out=225,in= 15] (-5_2);
\draw (-4_3) to[out=135,in= -15] (-5_3); \draw (-4_3) to[out=165,in=-45] (-5_3);
\draw (-4_4) to[out=195,in=45] (-5_4); \draw (-4_4) to[out=225,in= 15] (-5_4);

\node at (-1,0.2) {\tiny$a$};
\node at (0,0.2) {\tiny$e$};
\node at (1,0.7) {\tiny$b$};
\node at (1,-0.7) {\tiny$b^2$};
\node at (2,1.4) {\tiny$ba$};
\node at (2,-1.4) {\tiny$b^2a$};
\node at (2.9,1.9) {\tiny$bab$};
\node at (2.9,0.5) {\tiny$bab^2$};
\node at (2.8,-0.5) {\tiny$b^2ab$};
\node at (2.8,-1.9) {\tiny$b^2ab^2$};

\node at (-2,0.7) {\tiny$ab$};
\node at (-2,-0.7) {\tiny$ab^2$};
\node at (-3,1.4) {\tiny$aba$};
\node at (-3,-1.4) {\tiny$ab^2a$};

\node at (-3.8,1.9) {\tiny$abab$};
\node at (-3.8,0.5) {\tiny$abab^2$};
\node at (-3.7,-0.5) {\tiny$ab^2ab$};
\node at (-3.8,-1.9) {\tiny$ab^2ab^2$};

\node at (-.5,-5){$L_{\infty}$};
\coordinate (0') at (0,-4);
\coordinate (1) at (1,-4);
\coordinate (2) at (2,-4);
\coordinate (3) at (3,-4);
\coordinate (4) at (4,-4);
\coordinate (-1') at (-1,-4);
\coordinate (-2) at (-2,-4);
\coordinate (-3) at (-3,-4);
\coordinate (-4) at (-4,-4);
\coordinate (-5) at (-5,-4);

\node at (0,-4.2){\tiny$0$};
\node at (1,-4.2){\tiny$1$};
\node at (2,-4.2){\tiny$2$};
\node at (3,-4.2){\tiny$3$};
\node at (4,-4.2){\tiny$4$};
\node at (-1,-4.2){\tiny$-1$};
\node at (-2,-4.2){\tiny$-2$};
\node at (-3,-4.2){\tiny$-3$};
\node at (-4,-4.2){\tiny$-4$};
\node at (-5,-4.2){\tiny$-5$};

\draw[fill=black] (0') circle [radius=2pt];
\draw[fill=black] (1) circle [radius=2pt];
\draw[fill=black] (2) circle [radius=2pt];
\draw[fill=black] (3) circle [radius=2pt];
\draw[fill=black] (4) circle [radius=2pt];
\draw[fill=black] (-1') circle [radius=2pt];
\draw[fill=black] (-2) circle [radius=2pt];
\draw[fill=black] (-3) circle [radius=2pt];
\draw[fill=black] (-4) circle [radius=2pt];
\draw[fill=black] (-5) circle [radius=2pt];
\node at (-5.5,-4){$\dots$};
\node at (4.5,-4){$\dots$};

\draw (-5)--(-4)--(-3)--(-2)--(-1')--(0')--(1)--(2)--(3)--(4);
\makeatletter
\tikzset{my loop/.style =  {to path={
  \pgfextra{\let\tikztotarget=\tikztostart}
  [looseness=8,min distance=8mm]
  \tikz@to@curve@path}
  }}  
\makeatletter 
\path (1) edge[my loop] node[above] {\tiny\emph{\color{blue} $2$}} (1);
\path (3) edge[my loop] node[above] {\tiny\emph{\color{blue} $4$}} (3);
\path (-2) edge[my loop] node[above] {\tiny\emph{\color{blue} $2$}} (-2);
\path (-4) edge[my loop] node[above] {\tiny\emph{\color{blue} $4$}} (-4);

\node at (0.5,-3.85){\tiny\emph{\color{blue} $2$}};
\node at (1.5,-3.85){\tiny\emph{\color{blue} $4$}};
\node at (2.5,-3.85){\tiny\emph{\color{blue} $4$}};
\node at (3.5,-3.85){\tiny\emph{\color{blue} $8$}};
\node at (-0.5,-3.85){\tiny\emph{\color{blue} $2$}};
\node at (-1.5,-3.85){\tiny\emph{\color{blue} $2$}};
\node at (-2.5,-3.85){\tiny\emph{\color{blue} $4$}};
\node at (-3.5,-3.85){\tiny\emph{\color{blue} $4$}};
\node at (-4.5,-3.85){\tiny\emph{\color{blue} $8$}};

\draw [dashed, gray] (0)--(0');
\draw [dashed, gray] (1,0.3)--(1,-0.4);
\draw [dashed, gray] (1,-0.9)--(1,-3.25);
\draw [dashed, gray] (1,-3.55)--(1);
\draw [dashed, gray] (2_1)--(2,-1.2);
\draw [dashed, gray] (2,-1.5)--(2);
\draw [dashed, gray] (3,0.3)--(3,-0.4);
\draw [dashed, gray] (3,-2.1)--(3,-3.25);
\draw [dashed, gray] (3,-3.55)--(3); 
\draw [dashed, gray] (4_1)--(4);
\draw [dashed, gray] (-1)--(-1');
\draw [dashed, gray] (-2,0.3)--(-2,-0.4);
\draw [dashed, gray] (-2,-0.9)--(-2,-3.25);
\draw [dashed, gray] (-2,-3.55)--(-2);
\draw [dashed, gray] (-3_1)--(-3,-1.2);
\draw [dashed, gray] (-3,-1.5)--(-3);
\draw [dashed, gray] (-4,0.3)--(-4,-0.4);
\draw [dashed, gray] (-4,-2.1)--(-4,-3.25);
\draw [dashed, gray] (-4,-3.55)--(-4); 
\draw [dashed, gray] (-5_1)--(-5);
\end{tikzpicture}
\caption{$\Gamma(C_2*C_3)$ covers the line $L_\infty$.} \label{fig:covering}
\end{figure}

In this graphical realization, we can see that $\Gamma$ is invariant under the action of $C_2$ given by the reflection with respect to the central horizontal axis passing through the vertices $a$ and~$e$.
\begin{figure}
\center
\begin{tikzpicture}[scale=1.2]

\coordinate (0) at (0,0);
\coordinate (1_1) at (1,0);
\coordinate (2_1) at (2,0);

\begin{scope}[shift={(2,0)}]
\coordinate (3_1) at (1,0.5);
\coordinate (3_2) at (1,-0.5);
\coordinate (4_1) at (2,1.2);
\coordinate (4_2) at (2,-1);
\node at (2.5,1.2){$\dots$};
\node at (2.5,-1){$\dots$};
\end{scope}

\coordinate (-1) at (-1,0);
\begin{scope}[shift={(-1,0)}]
\coordinate (-2_1) at (-1,0);
\coordinate (-3_1) at (-2,0);
\end{scope}

\begin{scope}[shift={(-3,0)}]
\coordinate (-4_1) at (-1,0.5);
\coordinate (-4_2) at (-1,-0.5);
\coordinate (-5_1) at (-2,1.2);
\coordinate (-5_2) at (-2,-1);
\node at (-2.5,1.2){$\dots$};
\node at (-2.5,-1){$\dots$};
\end{scope}

\draw[fill=black] (0) circle [radius=2pt];
\draw[fill=black] (1_1) circle [radius=2pt];

\draw[fill=black] (2_1) circle [radius=2pt];

\draw[fill=black] (3_1) circle [radius=2pt];
\draw[fill=black] (3_2) circle [radius=2pt];

\draw[fill=black] (4_1) circle [radius=2pt];
\draw[fill=black] (4_2) circle [radius=2pt];

\draw[fill=black] (-1) circle [radius=2pt];
\draw[fill=black] (-2_1) circle [radius=2pt];

\draw[fill=black] (-3_1) circle [radius=2pt];

\draw[fill=black] (-4_1) circle [radius=2pt];
\draw[fill=black] (-4_2) circle [radius=2pt];

\draw[fill=black] (-5_1) circle [radius=2pt];
\draw[fill=black] (-5_2) circle [radius=2pt];

\draw (0)--(1_1)--(2_1)--(3_1)--(3_2)--(2_1);
\draw (3_1)--(4_1);
\draw (3_2)--(4_2);
\draw (0)--(-1)--(-2_1)--(-3_1)--(-4_1)--(-4_2)--(-3_1);
\draw (-4_1)--(-5_1);
\draw (-4_2)--(-5_2);

\makeatletter
\tikzset{my loop/.style =  {to path={
  \pgfextra{\let\tikztotarget=\tikztostart}
  [looseness=8,min distance=8mm]
  \tikz@to@curve@path}
  }}  
\makeatletter 
\path (1_1) edge[my loop] node[above] {\tiny\emph{\color{blue} $2$}} (1_1);
\path (-2_1) edge[my loop] node[above] {\tiny\emph{\color{blue} $2$}} (-2_1);

\node at (0.5,0.15){\tiny\emph{\color{blue} $2$}};
\node at (1.5,0.15){\tiny\emph{\color{blue} $4$}};
\node at (2.5,0.45){\tiny\emph{\color{blue} $2$}};
\node at (2.5,-0.45){\tiny\emph{\color{blue} $2$}};
\node at (3.15,0){\tiny\emph{\color{blue} $2$}};
\node at (3.4,1){\tiny\emph{\color{blue} $4$}};
\node at (3.4,-0.9){\tiny\emph{\color{blue} $4$}};

\node at (-0.5,0.15){\tiny\emph{\color{blue} $2$}};
\node at (-1.5,0.15){\tiny\emph{\color{blue} $2$}};
\node at (-2.5,0.15){\tiny\emph{\color{blue} $4$}};
\node at (-3.5,0.45){\tiny\emph{\color{blue} $2$}};
\node at (-3.5,-0.45){\tiny\emph{\color{blue} $2$}};
\node at (-4.15,0){\tiny\emph{\color{blue} $2$}};
\node at (-4.4,1){\tiny\emph{\color{blue} $4$}};
\node at (-4.4,-0.9){\tiny\emph{\color{blue} $4$}};
\end{tikzpicture}
\caption{The quotient weighted graph $\Gamma(C_2*C_3)/C_2$.} \label{fig:quotient}
\end{figure}
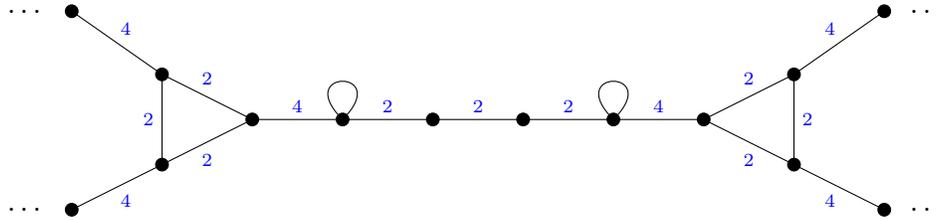

The resulting quotient graph $\Gamma/C_2$ is invariant under another action of $C_2$ given by the reflection with respect to its own central horizontal axis, see Fig.~\ref{fig:quotient}. By continuing similarly, we see that every resulting quotient graph is again invariant under a certain action of $C_2$. Hence, $L_{\infty}$ is the quotient weighted graph of $\Gamma$ with respect to the infinite group given by the product of infinitely many $C_2$'s, $G=C_2^{\times \infty}=C_2\times C_2\times\dotsm$.  It is also clear that the orbits of this action are finite. Therefore, by Proposition \ref{Prop2}, we conclude that $\Gamma$ covers $L_{\infty}$.

Moreover, this covering is strong and regular, where the distinguished vertex is $-1$ or $0$. Thus, our strategy is to first solve the projected spectral problem on $\Z$, which will yield an explicit formula for the heat kernel on $\Z$. We will then use formula \eqref{eq6} to obtain the heat kernel on $\Gamma$.

\begin{remark}
Note that $\Gamma$ is also invariant under the action of $C_2$ given by the reflection with respect to the vertical axis passing between vertices $a$ and $e$, but in this case there are no one-element orbits. 
\end{remark}

\subsection{Explicit expression for the $\pi$-projected Laplacian $\mathcal{L}^{pr}$ on $\Z$}

Let $\pi:G\to \Z$ be the Chung--Yau covering defined above.
Since there is an automorphism of $\Z$ such that $n\mapsto-n-1$ for $n\in\Z$, it will induce an operator of order two that will commute with the $\pi$-projected Laplacian on $\Z$.

Define, for $n\geq0$, $\alpha_n:=\vert \pi^{-1}(2n)\vert$ and $\beta_n:=\vert \pi^{-1}(2n-1)\vert$. Then, by definition of the covering, we have
$\alpha_0=1$, $\alpha_n=\beta_n$, $\beta_{n}=2\alpha_{n-1}$. 
Thus, we have $\alpha_{n}=2\alpha_{n-1}=2^{n}\alpha_0=2^{n}$ so
$$
\beta_n=\alpha_{n}=2^{n}.
$$
The weight function $w$ is defined, for $n\geq0$, by
\begin{equation}\label{w_L^pr}
\begin{cases}{ }
 w(2n,2n)=0 \\
 w(2n+1,2n+1)=2^{n+1} \\ 
 w(2n-1,2n)=w(2n,2n+1)=2^{n+1} 
\end{cases}
\end{equation}
and all other weights are 0.
The degree of a vertex $m\geq0$ is
$$ d_m=
\begin{cases}
 2^{n+2} & \text{if } m =2n \\
 2^{n+3} & \text{if } m=2n+1.
\end{cases}
$$
Thus, the non-zero matrix coefficients of the normalized Laplacian \eqref{normLaplacian}, for $n\geq0$, are 
\begin{equation}\label{Lpr}
\begin{cases}
 \mathcal{L}^{pr}(2n,2n)=1 \\
 \mathcal{L}^{pr}(2n+1,2n+1)=\frac{3}{4} \\
 \mathcal{L}^{pr}(2n-1,2n)=-\frac{1}{2} \\
 \mathcal{L}^{pr}(2n,2n+1)=-\frac{1}{2\sqrt{2}}.
\end{cases}
\end{equation}

As a consequence of the above observations, we obtain the following.
\begin{proposition}
 The $\pi$-projected normalized Laplacian $\mathcal{L}^{pr}$ on $\Z$ explicitly acts on a function $f:\Z \to \C$, for $m\geq 0$, as
\begin{equation} \label{Lp}
 (\mathcal{L}^{pr}f)(m)=
 \begin{cases} 
    f(m)-\frac{1}{2}f(m-1)-\frac{1}{2\sqrt{2}}f(m+1) &\text{if } m \text{ is even}  \\
   \frac{3}{4}f(m)-\frac{1}{2\sqrt{2}}f(m-1)-\frac{1}{2}f(m+1) &\text{if } m \text{ is odd}
\end{cases}
\end{equation}
and using the symmetry $m\mapsto -m-1$,
\begin{equation} \label{Lm}
 (\mathcal{L}^{pr}f)(-m-1)=
 \begin{cases} 
    f(-m-1)-\frac{1}{2}f(-m)-\frac{1}{2\sqrt{2}}f(-m-2) &\text{if } m \text{ is even}  \\
   \frac{3}{4}f(-m-1)-\frac{1}{2\sqrt{2}}f(-m)-\frac{1}{2}f(-m-2) &\text{if } m \text{ is odd}.
\end{cases}
\end{equation}
\end{proposition}
We note that $\mathcal{L}^{pr}$ is a self-adjoint operator in the Hilbert space $\ell^2(\Z)$.
\subsection{Solving the spectral problem for $\mathcal{L}^{pr}$}
In this section, we solve the spectral problem in the Hilbert space $\ell^2(\Z)$ for the self-adjoint operator $\mathcal{L}^{pr}$.

Define, for $n\geq0$, $u_n:=f(2n), v_n:=f(2n+1)$ and $\check{u}_n:=f(-2n-1), \check{v}_n:=f(-2n-2)$. Then, with the new notations, the eigenvalue problem $\mathcal{L}^{pr}f=\lambda f$ with formula \eqref{Lp}, becomes
 \begin{numcases} { }
2(1-\lambda)u_n=v_{n-1}+\frac{1}{\sqrt{2}}v_n \label{l1} \\
2(\frac{3}{4}-\lambda)v_n=\frac{1}{\sqrt{2}}u_{n}+u_{n+1} \label{l2}
\end{numcases}
while with formula \eqref{Lm} it becomes
$$
\begin{cases}
2(1-\lambda)\check{u}_n=\check{v}_{n-1}+\frac{1}{\sqrt{2}}\check{v}_n \label{l1m} \\
2(\frac{3}{4}-\lambda)\check{v}_n=\frac{1}{\sqrt{2}}\check{u}_{n}+\check{u}_{n+1}.  \label{l2m}
\end{cases}
$$
We also have the following sewing equations
\begin{equation}\label{se}
 v_{-1}=f(-1)=\check{u}_0 \quad \text{and} \quad \check{v}_{-1}=f(0)=u_0.
\end{equation}

First, assume that $\lambda=1$. Then, from equation \eqref{l1}, we get that 
$$
v_n=-\sqrt{2}v_{n-1}=(-\sqrt{2})^n v_0
$$
which immediately implies that $v_n=0$ for all $n\ge-1$, otherwise $v_n$ would not be square-summable. For the same reason, we conclude that $\check{v}_n=0$ for all $n\ge-1$.
By substituting $v_n=0$ in equation \eqref{l2}, we get that 
\begin{equation}
 u_{n+1}=-\frac{1}{\sqrt{2}}u_n=\Big(-\frac{1}{\sqrt{2}}\Big)^{n+1}u_0. \label{UN}
\end{equation}
Using the second sewing equation in \eqref{se}, we obtain $u_0=0$. Therefore, from equation \eqref{UN}, we conclude that $u_n=0$ for all $n\ge0$.
Similarly, we conclude that $\check{u}_n=0$ for all $n\ge0$. Therefore, $\lambda=1$ is not an eigenvalue of $\mathcal{L}^{pr}$.

Thus, we suppose that $\lambda \neq 1$. Multiplying equation~\eqref{l2} by $2(1-\lambda)$ and using equation~\eqref{l1}, we obtain
\begin{align*}
& 4(1-\lambda)(\frac{3}{4}-\lambda)v_n=\frac{1}{\sqrt{2}}\Big(v_{n-1}+\frac{1}{\sqrt{2}}v_{n}\Big)+v_{n}+\frac{1}{\sqrt{2}}v_{n+1}=\frac{3}{2}v_n+\frac{1}{\sqrt{2}}(v_{n-1}+v_{n+1}) \\
& \iff \sqrt{2}\Big(4\lambda^2-7\lambda+\frac{3}{2}\Big)v_n=v_{n-1}+v_{n+1}
\end{align*}
which is a linear system with constant coefficients. We use the standard approach by substituting $v_n=\xi^n$ and obtain the spectral equation
\begin{equation}\label{speceq1}
 \sqrt{2}\Big(4\lambda^2-7\lambda+\frac{3}{2}\Big)=\xi^{-1}+\xi.
\end{equation}
Both $\xi^{-1}$ and $\xi$ are solutions of \eqref{speceq1} and therefore the general solution of the equation is of the form $v_n=\alpha\xi^n+\beta\xi^{-n}$ with $\alpha,\beta\in\C$, where $\xi$ is defined such that equation~\eqref{speceq1} is satisfied.
The case $\vert \xi \vert>1$ corresponds to eigenvalues, while the case $\vert \xi \vert=1$ corresponds to continuous spectrum.
So we have
\begin{equation}
 v_n=\alpha \xi^n + \beta \xi^{-n}
\end{equation}
and from equation~\eqref{l1}, we have 
\begin{equation}
u_n=\frac{v_{n-1}+\frac{1}{\sqrt{2}}v_n}{2(1-\lambda)}. \label{un}
\end{equation}

\vskip 0.3cm
{\it The case $\vert \xi \vert>1$.}
Assume, without loss of generality, that $\vert \xi \vert>1$ (since the spectral equation~\eqref{speceq1} is invariant under the change $\xi \leftrightarrow \xi^{-1}$). 
Then, since we are looking for eigenfunctions $f\in\ell^2(\Z)$, we need to have $\alpha=0$, 
otherwise $f$ would not be a square-summable function. Thus, we have
\begin{equation}\label{uv}
 v_n=\beta \xi^{-n}, \quad u_n=\gamma \xi^{-n}
\end{equation} 
since $u_n$ is a linear combination of $v_{n-1}$ and $v_n$, and we also have the same expressions for $\check{v}_n$ and $\check{u}_n$ with the changes $\beta \leftrightarrow \check{\beta}, \gamma \leftrightarrow \check{\gamma}$
\begin{equation}\label{ouv}
 \check{v}_n=\check{\beta} \xi^{-n}, \quad \check{u}_n=\check{\gamma} \xi^{-n}.
\end{equation} 
Therefore, the sewing equations \eqref{se} are equivalent to 
\begin{equation}\label{nse}
 \beta \xi=\check{\gamma} \quad \text{and} \quad \check{\beta} \xi=\gamma.
\end{equation}
Using expression \eqref{uv} in equations \eqref{l1} and \eqref{l2}, we obtain
$$
\begin{cases} { }
2(1-\lambda)\gamma=\beta \Big(\xi+\frac{1}{\sqrt{2}}\Big) \\
2(\frac{3}{4}-\lambda)\beta=\gamma\Big(\xi^{-1}+\frac{1}{\sqrt{2}}\Big)
\end{cases}
$$
and using the second sewing equation $\gamma= \check{\beta}\xi$, we obtain 
\begin{equation}\label{nl1l2}
 \begin{cases} { }
2(1-\lambda)=\frac{\beta}{ \check{\beta}} \Big(1+\frac{1}{\sqrt{2}\xi}\Big) \\
2(\frac{3}{4}-\lambda)=\frac{\check{\beta}}{\beta}\Big(1+\frac{\xi}{\sqrt{2}}\Big)
\end{cases}
\end{equation}
which is consistent with the spectral equation \eqref{speceq1}.
Furthermore, using expression~\eqref{uv} for $v_n$ in expression~\eqref{un}, we have 
$$
u_n=\frac{1}{2(1-\lambda)}\Big(v_{n-1}+\frac{1}{\sqrt{2}}v_n\Big)=\frac{\beta}{2(1-\lambda)}\Big(\xi^{1-n}+\frac{1}{\sqrt{2}}\xi^{-n}\Big)
$$
and we also obtain the same expression for $\check{u}_n$ with the change $\beta \leftrightarrow \check{\beta}$.
Thus, the sewing equations~\eqref{se} are equivalent to 
$$
\begin{cases} { }
\beta \xi=\frac{\check{\beta}}{2(1-\lambda)}\Big(\xi+\frac{1}{\sqrt{2}}\Big) \\
\check{\beta} \xi=\frac{\beta}{2(1-\lambda)}\Big(\xi+\frac{1}{\sqrt{2}}\Big)
\end{cases}
\iff 
\frac{\beta}{ \check{\beta}}=\frac{\check{\beta}}{\beta}
\iff
\Big(\frac{\beta}{ \check{\beta}}\Big)^2=1.
$$
Let us denote by 
\begin{equation} \label{def_epsilon}
 \epsilon:=\frac{\beta}{ \check{\beta}}\in\{\pm1\}.
\end{equation}
Then, system~\eqref{nl1l2} is equivalent to 
\begin{equation}\label{fl1l2}
  \begin{cases} { }
2\epsilon(1-\lambda)=1+\frac{1}{\sqrt{2}\xi} \\
2\epsilon(\frac{3}{4}-\lambda)=1+\frac{\xi}{\sqrt{2}}
\end{cases}
\end{equation}
and by subtracting the second equation from the first, we obtain 
$$
\frac{\epsilon}{2}=\frac{1}{\sqrt{2}}(\xi^{-1}-\xi) \iff \xi^2+\frac{\epsilon}{\sqrt{2}}\xi=1
$$
that gives us four different solutions 
$$
\xi_{\epsilon,\pm}=\frac{-\epsilon\pm3}{2\sqrt{2}}.
$$
Since we are looking for a solution $\xi$ such that $\vert \xi \vert >1$, we only keep the solutions satisfying this condition, one for $\epsilon=1$ and one for $\epsilon=-1$
$$
\xi_\epsilon=-\epsilon\sqrt{2}.
$$ 
Using this solution in \eqref{fl1l2}, we obtain two discrete eigenvalues, one for $\epsilon=1$ and one for $\epsilon=-1$,  
$$
2\epsilon(1-\lambda)=1-\frac{\epsilon}{2} \iff \lambda_\epsilon=\frac{5-2\epsilon}{4}=
\begin{cases}
 \frac{3}{4} & \text{if }\epsilon=1 \\
 \frac{7}{4} & \text{if }\epsilon=-1.
\end{cases}
$$
The corresponding eigenfunctions are given by
$$
f_\epsilon(m)=
\begin{cases}
 \gamma \xi_\epsilon^{-n}=\epsilon\beta\xi_\epsilon^{1-n} & \text{if } m=2n \\
 \beta\xi_\epsilon^{-n}& \text{if } m=2n+1 \\
 \check{\gamma} \xi_\epsilon^{-n}=\beta\xi_\epsilon^{1-n} & \text{if } m=-2n-1 \\
 \check{\beta} \xi_\epsilon^{-n}=\epsilon\beta\xi_\epsilon^{-n} & \text{if } m=-2n-2
\end{cases},
\quad \text{where } \xi_\epsilon=-\epsilon\sqrt{2}
$$
and we used expressions \eqref{uv}, \eqref{ouv}, \eqref{nse} and \eqref{def_epsilon}.
We compute the norm of $f_\epsilon$ in order to fix the remaining free parameter $\beta$
$$
\lVert f_\epsilon \rVert^2=\sum_{n=0}^\infty \big( \vert u_n\vert^2 + \vert v_n\vert^2 + \vert\overline{u}_n\vert^2 + \vert\overline{v}_n\vert^2\big)
=2\vert\beta\vert^2 (\vert\xi_\epsilon\vert^2+1) \sum_{n=0}^\infty \vert\xi_\epsilon\vert^{-2n} 
=12\vert\beta\vert^2
$$
where, in the last equality, we used
 the fact that $\vert\xi_\epsilon\vert^2=2$.
Therefore, by choosing $\beta=\frac{1}{2\sqrt{3}}$, we obtain eigenfunctions $f_\epsilon$ of norm $1$ defined as
$$
f_\epsilon(m)=
\begin{cases}
 \frac{-1}{\sqrt{6}}(-\epsilon\sqrt{2})^{-n} & \text{if } m=2n \\
 \frac{1}{2\sqrt{3}}(-\epsilon\sqrt{2})^{-n} & \text{if } m=2n+1 \\
 \frac{-\epsilon}{\sqrt{6}}(-\epsilon\sqrt{2})^{-n} & \text{if } m=-2n-1 \\
 \frac{\epsilon}{2\sqrt{3}}(-\epsilon\sqrt{2})^{-n} & \text{if } m=-2n-2.
\end{cases}
$$

\vskip 0.3cm
{\it The case $\vert \xi \vert=1$.}
Finally, assume that $\vert \xi \vert=1$, that is $\xi=e^{ix}$, where, using the symmetry of the spectral equation under $\xi \leftrightarrow \xi^{-1}$, we can always assume that $x\in[0,\pi]$. We rewrite the spectral equation~\eqref{speceq1} as 
$$
4(1-\lambda)\Big(\frac{3}{4}-\lambda\Big)-\frac{3}{2}=\frac{1}{\sqrt{2}}(\xi^{-1}+\xi)=\frac{1}{\sqrt{2}}(e^{-ix}+e^{ix})=\sqrt{2}\cos(x).
$$
Introducing a new variable $\nu$ through $2(1-\lambda)=\nu+\frac{1}{4}$, we have $2(\frac{3}{4}-\lambda)=\nu-\frac{1}{4}$, which converts the previous equality to
\begin{equation}\label{mu1}
 \nu^2-\frac{1}{16}=\frac{3}{2}+\sqrt{2}\cos(x)
\iff
\nu^2=\frac{25}{16}+\sqrt{2}\cos(x).
\end{equation}
Therefore, we have 
\begin{equation}
 \nu_{\mu,x}=\mu R_x, \quad \text{where } R_x:=\sqrt{\frac{25}{16}+\sqrt{2}\cos(x)}>0, \ \mu\in\{\pm 1\}.
\end{equation}
Note that $R_x=R_{-x}$. By recalling the definition above of $\nu$ in terms of $\lambda$, we obtain an expression for $\lambda$
$$
\lambda_{\mu,x}=\frac{7}{8}-\frac{\nu_{\mu,x}}{2}=\frac{7}{8}-\frac{\mu}{2}\sqrt{\frac{25}{16}+\sqrt{2}\cos(x)}, \quad \text{where} \ x\in[0,\pi].
$$
Recall that 
\begin{equation}\label{newv}
 v_n=\alpha \xi^n + \beta \xi^{-n}, \quad \alpha,\beta\in\C
\end{equation}
and since $u_n$ is a linear combination of $v_{n-1}$ and $v_n$, we have
\begin{equation}\label{newu}
 u_n=\gamma \xi^n + \delta \xi^{-n}, \quad \gamma,\delta\in\C.
\end{equation}
Consider the operator $P$ defined by its action on a function $f:\Z \to \C$, by 
$$
Pf(n):=f(-n-1).
$$
Since $P^2=\operatorname{Id}$ and $P$ commutes with $\mathcal{L}^{pr}$, we can always assume that a generalized eigenfunction is also an eigenfunction of $P$, and we have 
\begin{equation}
 Pf=\epsilon f, \quad \text{with } \epsilon\in\{\pm1\}. \label{epsilon}
\end{equation}
Thus, using \eqref{epsilon}, we can write $\check{v}_n$ as
\begin{equation}\label{nov}
\check{v}_n=Pf(2n+1)=\epsilon v_n=\epsilon(\alpha \xi^n + \beta \xi^{-n})
\end{equation}
and similarly, we can also write $\check{u}_n$ as
\begin{equation}\label{nou}
\check{u}_n=Pf(2n)=\epsilon u_n=\epsilon(\gamma \xi^n + \delta \xi^{-n}).
\end{equation}
Let us use \eqref{newv} and \eqref{newu} to rewrite \eqref{l1} and \eqref{l2} in terms of $\nu$ and $\xi$.
From \eqref{l1}, we have
$$
\Big(\nu+\frac{1}{4}\Big)(\gamma \xi^n + \delta \xi^{-n})=\alpha \xi^{n-1} + \beta \xi^{1-n}+\frac{1}{\sqrt{2}}(\alpha \xi^n + \beta \xi^{-n})
$$
and by equaling the coefficients of $\xi^n$ and $\xi^{-n}$, we obtain 
\begin{equation}\label{gd}
 \begin{cases} 
      (\nu+\frac{1}{4})\gamma=(\xi^{-1}+\frac{1}{\sqrt{2}})\alpha
      \\
      (\nu+\frac{1}{4})\delta=(\xi+\frac{1}{\sqrt{2}})\beta
\end{cases}
\iff 
\begin{cases} 
      \gamma=\frac{\xi^{-1}+\frac{1}{\sqrt{2}}}{\nu+\frac{1}{4}}\alpha
      \\
      \delta=\frac{\xi+\frac{1}{\sqrt{2}}}{\nu+\frac{1}{4}}\beta.
\end{cases}
\end{equation}
Similarly, from \eqref{l2}, we obtain expressions for $\alpha$ and $\beta$ 
$$
\Big(\nu-\frac{1}{4}\Big)(\alpha \xi^n + \beta \xi^{-n})=\gamma \xi^{n+1} + \delta \xi^{-n-1}+\frac{1}{\sqrt{2}}(\gamma \xi^n + \delta \xi^{-n})
$$
and by equaling the coefficients of $\xi^n$ and $\xi^{-n}$, we obtain 
\begin{equation}\label{ab}
\begin{cases} 
      (\nu-\frac{1}{4})\alpha=(\xi+\frac{1}{\sqrt{2}})\gamma
      \\
      (\nu-\frac{1}{4})\beta=(\xi^{-1}+\frac{1}{\sqrt{2}})\delta
\end{cases}
\iff 
\begin{cases} 
      \alpha=\frac{\xi+\frac{1}{\sqrt{2}}}{\nu-\frac{1}{4}}\gamma
      \\
      \beta=\frac{\xi^{-1}+\frac{1}{\sqrt{2}}}{\nu-\frac{1}{4}}\delta.
\end{cases}
\end{equation}
By substituting these into expression \eqref{gd} for $\gamma$ and $\delta$, we obtain expression \eqref{mu1} for $\nu^2$, so expressions \eqref{gd} and \eqref{ab} are equivalent.
Using expressions \eqref{newv}, \eqref{newu}, \eqref{nov} and \eqref{nou}, the sewing equations \eqref{se}  reduce to one relation
\begin{equation}\label{nse12}
 \epsilon(\gamma+\delta)=\alpha\xi^{-1}+\beta\xi 
\end{equation}
which, by using \eqref{gd}, can be rewritten  as  follows:
\begin{align}
 & \epsilon\Bigg(\frac{\xi^{-1}+\frac{1}{\sqrt{2}}}{\nu+\frac{1}{4}}\alpha+\frac{\xi+\frac{1}{\sqrt{2}}}{\nu+\frac{1}{4}}\beta\Bigg)=\alpha\xi^{-1}+\beta\xi \notag \\
 \iff
& \alpha\Big(\frac{\epsilon}{\sqrt{2}}-\xi^{-1}(\nu+\frac{1}{4}-\epsilon)\Big)=\beta\Big(\xi(\nu+\frac{1}{4}-\epsilon)-\frac{\epsilon}{\sqrt{2}}\Big) \notag \\
 \iff
& \frac{\alpha}{\beta}=-\frac{\xi(\nu+\frac{1}{4}-\epsilon)-\frac{\epsilon}{\sqrt{2}}}{\xi^{-1}(\nu+\frac{1}{4}-\epsilon)-\frac{\epsilon}{\sqrt{2}}}. \label{al/be}
\end{align}
Observe that the denominator is the complex conjugate of the numerator. Since the coefficients $\alpha, \beta,\gamma,\delta$ are determined up to a common multiplicative factor, we can choose $\alpha$ arbitrarily and $\beta,\gamma,\delta$ are then automatically determined from~\eqref{al/be} and \eqref{gd}. Let us  choose 
\begin{equation}
 \alpha:=\frac1{2i}\big(\xi(\nu+\frac{1}{4}-\epsilon)-\frac{\epsilon}{\sqrt{2}}\big), \quad \epsilon\in\{\pm1\}.
\end{equation}
Then, from~\eqref{al/be} we obtain 
\begin{equation}
\beta=\overline{\alpha}
\end{equation}
where $\overline{\alpha}$ is the complex conjugate of $\alpha$, and from \eqref{gd} we obtain 
\begin{equation}
\gamma=\frac1{2i}\big(1+\epsilon(\frac{1}{4}-\epsilon)+\frac{\xi}{\sqrt{2}}\big)
\end{equation}
and 
\begin{equation}
\delta=\overline{\gamma}.
\end{equation}
With this normalisation,  with $x\in[0,\pi]$, $\mu, \epsilon \in\{\pm 1\}$, $n\ge0$, the generalized eigenvectors are real
\begin{align*}
  f_{x,\mu, \epsilon}(2n) & =u_n=\gamma\xi^n+\delta\xi^{-n}=\gamma\xi^n+\overline{\gamma}\xi^{-n}=2\operatorname{Re} (\gamma\xi^n)=\operatorname{Im}\Big((1+\epsilon(\frac{1}{4}-\nu))\xi^n+\frac{\xi^{n+1}}{\sqrt{2}}\Big)\\
 &=(1+\frac{\epsilon}{4}-\epsilon\mu R_x)\sin(nx)+\frac{1}{\sqrt{2}}\sin((n+1)x),
\end{align*}
where in the last equality we used the fact that $\xi=e^{ix}$ and $\nu=\mu R_x$, and similarly
\begin{align*}
  f_{x,\mu, \epsilon}(2n+1) & =v_n=\alpha\xi^n+\beta\xi^{-n}=\alpha\xi^n+\overline{\alpha}\xi^{-n} \\
 &=(\mu R_x+\frac{1}{4}-\epsilon)\sin((m+1)x)-\frac{\epsilon}{\sqrt{2}}\sin(nx) ,
\end{align*}
\begin{equation*}
  f_{x,\mu, \epsilon}(-2n-1)=\epsilon f_{x,\mu, \epsilon}(2n), \quad f_{x,\mu, \epsilon}(-2n-2)=\epsilon f_{x,\mu, \epsilon}(2n+1).
\end{equation*}
The corresponding generalized eigenvalues are
$$
\lambda_{\mu,x}=\frac{7}{8}-\frac{\mu}{2}\sqrt{\frac{25}{16}+\sqrt{2}\cos(x)}.
$$
The following proposition is a summary of what we have done up to now.
\begin{proposition}\label{spectrum_L^pr}
 The $\pi$-projected Laplacian $\mathcal{L}^{pr}$ on $\Z$ has two discrete eigenvalues, parameterized by $\epsilon\in\{\pm 1\}$,  
$$
 \lambda_\epsilon=\frac{5-2\epsilon}{4}=
\begin{cases}
 \frac{3}{4} & \text{if }\epsilon=1 \\
 \frac{7}{4} & \text{if }\epsilon=-1
\end{cases}
$$
with the corresponding real eigenvectors possessing the following symmetry property
$$
f_\epsilon(-m-1)=\epsilon f_\epsilon(m), \quad \forall m\in\Z.
$$
For $m\ge0$,
\begin{numcases}{f_\epsilon(m)=}
 \frac{-1}{\sqrt{6}}(-\epsilon\sqrt{2})^{-n} & if $m=2n$ \label{dfpe} \\
 \frac{1}{2\sqrt{3}}(-\epsilon\sqrt{2})^{-n} & if $m=2n+1$. \label{dfpo}
\end{numcases}
It also has the following generalized eigenvalues, parameterized by $x\in[0,\pi]$ and $\mu\in\{\pm 1\}$,
$$
\lambda_{\mu,x}=\frac{7}{8}-\frac{\mu}{2}R_x
$$
with 
$$
R_x=\sqrt{\frac{25}{16}+\sqrt{2}\cos(x)},
$$
where the corresponding (generalized) eigenspaces are two dimensional so that one can choose a basis indexed by variable $\epsilon\in\{\pm1\}$ so that
$$
f_{x,\mu, \epsilon}(-m-1)=\epsilon f_{x,\mu, \epsilon}(m), \quad \forall m\in\Z.
$$
For $m\ge0$,
\begin{numcases}{f_{x,\mu, \epsilon}(m)=}
 (1+\frac{\epsilon}{4}-\epsilon\mu R_x)\sin(nx)+\frac{1}{\sqrt{2}}\sin((n+1)x) & if $m=2n$ \label{cfpe} \\
 (\mu R_x+\frac{1}{4}-\epsilon)\sin((n+1)x)-\frac{\epsilon}{\sqrt{2}}\sin(nx) & if $m=2n+1$. \label{cfpo}
\end{numcases}
\end{proposition}

The scalar products of the (generalized) eigenfunctions, interpreted in a distributional sense, allow us to determine the spectral measure associated with this operator.
\begin{proposition}\label{scalar_products1}
We have the following scalar products, for $\mu,\mu',\epsilon, \epsilon' \in \{\pm 1\}$ and $x,x'\in~[0,\pi]$,
 $$
 \langle f_{x,\mu,\epsilon},  f_{x',\mu',\epsilon'} \rangle=H_{\mu\epsilon}(x)\delta_{\mu,\mu'}\delta_{\epsilon,\epsilon'}\delta(x-x')
$$
where 
\begin{equation}\label{H_eps}
 H_{\nu}(x):=2\pi R_x\big(2R_x-\nu(2+\sqrt{2}\cos(x))\big)>0, \quad \nu=\pm1,
\end{equation}
and $\delta(x)$ is Dirac's delta function.
\end{proposition}
\begin{proof}
 
Let $\mu,\mu',\epsilon, \epsilon' \in \{\pm 1\}$ and $x,x' \in [0,\pi]$.
\begin{align*}
 \langle f_{x,\mu,\epsilon}, f_{x',\mu',\epsilon'} \rangle 
& = \sum_{m=0}^\infty \Big( f_{x,\mu,\epsilon}(m) f_{x',\mu',\epsilon'}(m) + \epsilon\epsilon' f_{x,\mu,\epsilon}(m) f_{x',\mu',\epsilon'}(m) \Big) \\
& = (1+\epsilon\epsilon') \sum_{n=0}^\infty \Big(f_{x,\mu,\epsilon}(2n) f_{x',\mu',\epsilon'}(2n) + f_{x,\mu,\epsilon}(2n+1) f_{x',\mu',\epsilon'}(2n+1)\Big) \\
& = 2\delta_{\epsilon,\epsilon'} \sum_{n=0}^\infty \Big( f_{x,\mu,\epsilon}(2n) f_{x',\mu',\epsilon}(2n) + f_{x,\mu,\epsilon}(2n+1) f_{x',\mu',\epsilon}(2n+1)\Big)
 \end{align*}
 where, in the first equality, we used the fact that the eigenvectors are real and in the last equality, the fact that 
 $$
 \epsilon\epsilon'=
\begin{cases}
 1 & \text{if } \epsilon=\epsilon' \\
 -1 & \text{if } \epsilon\neq\epsilon'
\end{cases}
$$
since $\epsilon,\epsilon'\in\{\pm1\}$.

To compute the sum above, we start by using formulas \eqref{cfpe} and \eqref{cfpo} and we collect terms with different products of sine functions and simplify the coefficients 
\begin{itemize}
 \item $2\sin(nx)\sin(nx')$: $\big(1+\frac{\epsilon}{4}-\epsilon\mu R_x\big) \big(1+\frac{\epsilon}{4}-\epsilon\mu' R_{x'}\big)+ \frac{1}{2}=:c_1(x,x')$
 \item $2\sin((n+1)x)\sin((n+1)x')$: $\big(\mu R_x+\frac{1}{4}-\epsilon\big)\big(\mu' R_{x'}+\frac{1}{4}-\epsilon\big)+ \frac{1}{2}=:c_2(x,x')$
 \item $2\sin(nx)\sin((n+1)x')$: $\frac{1+\frac{\epsilon}{4}-\epsilon\mu R_x}{\sqrt{2}} - \frac{\epsilon\big(\mu' R_{x'}+\frac{1}{4}-\epsilon\big)}{\sqrt{2}}=\sqrt{2}-\frac{\epsilon}{\sqrt{2}}(\mu R_x+\mu' R_{x'})=:c_3(x,x')$ \\
 \item $2\sin(nx')\sin((n+1)x)$: $\frac{1+\frac{\epsilon}{4}-\epsilon\mu' R_{x'}}{\sqrt{2}} - \frac{\epsilon\big(\mu R_{x}+\frac{1}{4}-\epsilon\big)}{\sqrt{2}}=\sqrt{2}-\frac{\epsilon}{\sqrt{2}}(\mu R_x+\mu' R_{x'})=c_3(x,x')$ 
\end{itemize}
Thus, we rewrite the scalar product above as 
\begin{align}
  \langle f_{x,\mu,\epsilon}, f_{x',\mu',\epsilon'} \rangle 
 = \  \delta_{\epsilon,\epsilon'} \Big( & c_1(x,x')\sum_{n=0}^\infty 2\sin(nx)\sin(nx') + c_2(x,x')\sum_{n=0}^\infty 2\sin((n+1)x)\sin((n+1)x') \notag \\
  + \ & c_3(x,x')\sum_{n=0}^\infty \big(2\sin(nx)\sin((n+1)x') + 2\sin(nx')\sin((n+1)x) \big) \Big).\label{sp_continuous}
 \end{align}
 We compute each sum separately using the following trigonometric identity
\begin{equation} \label{2sin_id}
  2\sin(a)\sin(b)=\cos(a-b)-\cos(a+b).
\end{equation}
 First, notice that the first and the second sum are equal using a change of summation variable by $n \mapsto n-1$.
Using the distributional form of the Poisson summation formula,
\begin{equation}\label{poisson_summation}
 \sum_{n\in\Z}e^{2\pi inx}=\sum_{k\in\Z}\delta(x+k)
\end{equation}
and the equality
\begin{equation} \label{Diracdelta}
 f(x)\delta(x-y)=f(y)\delta(x-y)
\end{equation}
for any continuous function $f$,
we obtain the following distributional identities:
\begin{align} 
& \sum_{n=0}^\infty 2\sin(nx)\sin(nx') =\pi \delta(x-x') \label{2sin}, \\ 
& \sum_{n=0}^\infty \Big( 2\sin(nx) \sin((n+1)x') + 2\sin(nx')\sin((n+1)x) \Big) =2\pi \cos(x) \delta(x-x'), \label{lastsum_final}
\end{align}
where $x,x'\in[0,\pi]$, see Appendix~\ref{appA} for details.

Finally, using \eqref{2sin} and \eqref{lastsum_final}, expression \eqref{sp_continuous} becomes
\begin{align}
 \langle f_{x,\mu,\epsilon},  f_{x',\mu',\epsilon'} \rangle =\pi g(x,x')\delta_{\epsilon,\epsilon'}\delta(x-x')=\pi g(x)\delta_{\epsilon,\epsilon'}\delta(x-x') \label{sp_g(x)}
\end{align} 
where 
$$
 g(x,x'):= c_1(x,x')+c_2(x,x')+2\cos(x)c_3(x,x'),
$$
$g(x):=g(x,x)$
and, in the second equality of \eqref{sp_g(x)}, we used \eqref{Diracdelta} for $g$.
Let us simplify the expression for $g(x)$:
\begin{align*}
 g(x) = 2\mu\mu'R_x^2-\epsilon R_x(\mu+\mu')(2+\sqrt{2}\cos(x))+\frac{25}{8}+2\sqrt{2}\cos(x).
\end{align*}
Since $\mu,\mu'\in\{\pm 1\}$, either $\mu=-\mu'$ or $\mu=\mu'$:
\begin{itemize}
 \item if $\mu=-\mu'$, then $g(x)=-2R_x^2+\frac{25}{8}+2\sqrt{2}\cos(x)=0$
 \item if $\mu=\mu'$, then $g(x)=2R_x\big(2R_x-\epsilon\mu(2+\sqrt{2}\cos(x))\big)$,
 \end{itemize}
 where we used the fact that $R_x^2=\frac{25}{16}+\sqrt{2}\cos(x)$.
Therefore, we obtain
$$
g(x)=2R_x\big(2R_x-\epsilon\mu(2+\sqrt{2}\cos(x))\big)\delta_{\mu,\mu'}.
$$ 
Finally, from \eqref{sp_g(x)}, we have 
$$
 \langle f_{x,\mu,\epsilon},  f_{x',\mu',\epsilon'} \rangle=H_{\mu\epsilon}(x)\delta_{\mu,\mu'}\delta_{\epsilon,\epsilon'}\delta(x-x')
$$
where $H_{\nu}(x)$ is given in~\eqref{H_eps}.
\end{proof}

\begin{proposition}\label{completeness1}  
The following completeness condition is satisfied: for all $m,n\in\Z$,
$$
\sum_{\epsilon\in\{\pm1\}}f_\epsilon(m)f_\epsilon(n)+\int_0^\pi \sum_{\mu,\epsilon\in\{\pm1\}}f_{x,\mu,\epsilon}(m)f_{x,\mu,\epsilon}(n)\frac{\mathrm{d}x}{H_{\mu\epsilon}(x)}=\delta_{m,n}.
$$
\end{proposition}

\begin{proof}
 Let us denote
\begin{equation}
 F_x(m,n):=\sum_{\mu,\epsilon\in\{\pm1\}}f_{x,\mu,\epsilon}(m)f_{x,\mu,\epsilon}(n)\frac{1}{H_{\mu\epsilon}(x)} \label{F_x}
\end{equation}
 and
\begin{equation}
  C(m,n):=\int_0^\pi F_x(m,n) \ \mathrm{d}x +\sum_{\epsilon\in\{\pm1\}}f_\epsilon(m)f_\epsilon(n). \label{C_x}
\end{equation}
 We need to check separately the different cases depending on the parity of $m$ and $n$. 
The cases where both entries are negative are equivalent to the cases with both positive entries because of the symmetry property of the (generalized) eigenfunctions
$$
f_\epsilon(-m-1)=\epsilon f_\epsilon(m), \quad f_{x,\mu, \epsilon}(-m-1)=\epsilon f_{x,\mu, \epsilon}(m), \quad \forall m\in\Z.
$$
We will show in detail only the computation of $C(2m,2n+1)$ with $m,n\in\Z_{\ge0}$ since 
the computations of the other cases are analogous.
First, we compute the sum in \eqref{F_x} with the result
\begin{equation}\label{eq:positive-even-odd}
F_x(2m,2n+1)=\frac{\sqrt{2}}{\pi} \cdot \frac{2\big( \sin(mx)\sin(nx)-\sin((m+1)x)\sin((n+1)x)\big)}{4\cos(2x)-5}.
\end{equation}
Next, we transform  the numerator and denominator of~\eqref{eq:positive-even-odd} separately. The identity for the product of sines
$$
2\sin(x)\sin(y)=\cos(x-y)-\cos(x+y)
$$
allows us to rewrite the numerator of~\eqref{eq:positive-even-odd}  as
\begin{align}
 2\big( \sin(mx)\sin(nx) & -\sin((m+1)x)\sin((n+1)x)\big) = \cos(x(m+n+2))-\cos(x(m+n)) \notag \\
 & = \operatorname{Re}(e^{ix(m+n+2)}-e^{ix(m+n)}) = \operatorname{Re}(e^{ix(k+2)}-e^{ixk}), \label{num_ee}
\end{align}
where $k:=m+n$.
Denoting $\xi:=e^{2ix}$, we rewrite the denominator of~\eqref{eq:positive-even-odd}  as
$$
4\cos(2x)-5=2(\xi+\xi^{-1})-5=\frac{2}{\xi}(\xi-2)(\xi-\frac{1}{2}).
$$
Using the fraction decomposition for the inverse and the geometric series (expanding for $|\xi|=1$), we obtain
\begin{align}
 \frac{1}{4\cos(2x)-5} &=\frac{\xi/2}{(\xi-2)(\xi-\frac{1}{2})} 
=-\frac{1}{3}\Big(\frac{1}{1-\frac{\xi}{2}}+\frac{1}{2\xi}\frac{1}{1-\frac{1}{2\xi}}\Big) =-\frac{1}{3}\Big(\sum_{l=0}^{\infty}\big(\frac{\xi}{2}\big)^l + \sum_{l=0}^{\infty} \big(\frac{1}{2\xi}\big)^{l+1} \Big) \notag \\ 
&=-\frac{1}{3}\Big(\sum_{l=0}^{\infty}2^{-l}\xi^l + \sum_{l=-\infty}^{-1} (2\xi)^l \Big) 
=-\frac{1}{3} \sum_{l\in\Z}2^{-|l|}\xi^l=-\frac{1}{3} \sum_{l\in\Z}2^{-|l|}e^{2ixl}. \label{denom_ee}
\end{align}
Using \eqref{num_ee} and \eqref{denom_ee}, we can rewrite $F_x$ as follows:
\begin{align*}
 F_x(2m,2n+1)
 &=\frac{\sqrt{2}}{\pi} \operatorname{Re}\Big(\frac{1}{4\cos(2x)-5}(e^{ix(k+2)}-e^{ixk})\Big) \\
 &=-\frac{\sqrt{2}}{3\pi} \sum_{l\in\Z}2^{-|l|} \operatorname{Re}(e^{ix(2l+k+2)}-e^{ix(2l+k)}) \\
 &=-\frac{\sqrt{2}}{3\pi} \sum_{l\in\Z}2^{-|l|} \big( \cos(x(2l+k+2))-\cos(x(2l+k)) \big).
\end{align*}
Using the fact that 
$$
\int_0^\pi \cos(x(2l+k)) \mathrm{d}x =
\begin{cases}
\pi \delta_{l,-k/2} & k \ \text{even} \\
0 & \text{else},
\end{cases}
$$
we can now compute the integral (recall that $k=m+n$)
\begin{align*}
 \int_0^\pi F_x(2m,2n+1) \mathrm{d}x & =-\frac{\sqrt{2}}{3\pi} \sum_{l\in\Z}2^{-|l|} \int_0^\pi \cos(x(2l+k+2))-\cos(x(2l+k)) \mathrm{d}x \\
 &= -\frac{\sqrt{2}}{3\pi}
\begin{cases}
 \pi 2^{-(k+2)/2} - \pi 2^{-k/2} & k \ \text{even} \\
 0 & \text{else}
\end{cases}
\\
&= 
\begin{cases}
\frac{\sqrt{2}}{6} 2^{-k/2}  & k \ \text{even} \\
 0 & \text{else}.
 \end{cases}
\end{align*}
Finally, using the definition of the eigenfunctions $f_\epsilon(m)$, we compute the sum in \eqref{C_x}
\begin{align*}
 \sum_{\epsilon\in\{\pm1\}}f_\epsilon(2m)f_\epsilon(2n+1) & =\sum_{\epsilon\in\{\pm1\}}-\frac{1}{6\sqrt{2}}(-\sqrt{2}\epsilon)^{-k}=-\frac{1}{6\sqrt{2}}(-1)^k\sqrt{2}^{-k}(1+(-1)^k) \\
 &=
 \begin{cases}
  -\frac{\sqrt{2}}{6}\sqrt{2}^{-k} & k \ \text{even} \\
 0 & \text{else}.
\end{cases}
\end{align*}
Thus, we have shown that $C(2m,2n+1)= 0$.
\end{proof}

\subsection{Spectral theorem for the projected Laplacian}

 Let $J=J_c\sqcup J_d$, where $J_c:=[0,\pi]\times\{\pm1\}^2$ and $J_d:=\{\pm1\}$, be a measured space provided with the Borel $\sigma$-algebra $\mathcal{B}_J$ and the measure $\eta$ defined by 
 $$
 \eta(B)=\int_0^\pi \sum_{\mu,\epsilon\in\{\pm1\}} \chi_{B\cap J_c}(x,\mu,\epsilon) \frac{\mathrm{d}x}{H_{\mu\epsilon}(x)}+\sum_{\epsilon\in\{\pm1\}} \chi_{B\cap J_d}(\epsilon), \quad \forall B\in \mathcal{B}_J,
 $$
 where $H_\nu(x)$ is defined in~\eqref{H_eps}.
 Points in $J$ belong either to the continuous component $J_c$ or the discrete component $J_d$. Elements of $J_c$ are denoted by triplets $(x,\mu,\epsilon)$, while elements of $J_d$ are denoted by a single label $\epsilon$.
 Denote $L^2(J)$ the complex Hilbert space of square-integrable functions on $J$ (with respect to the measure $\eta$). 
 
 We define two functions on $J$ and a function on $J\times \Z$ as follows:
 $$
 \lambda:J\to\R
 $$
 defined by 
 $$
\lambda(x,\mu,\epsilon)=\lambda_{x,\mu}=\frac{7}{8}-\frac{\mu}{2}\sqrt{\frac{25}{16}+\sqrt{2}\cos(x)}, \quad \lambda(\epsilon)=\lambda_\epsilon=\frac{5-2\epsilon}{4},
 $$
 $$
 \sigma: J \to \R
 $$
 defined by
 $$
 \sigma(x,\mu,\epsilon)=\epsilon, \quad \sigma(\epsilon)=\epsilon
 $$
 and 
 $$
 f: J \times \Z \to \R
 $$
 defined, for $m\ge0$, by
\begin{align*}
& f(x,\mu,\epsilon,m)=f_{x,\mu,\epsilon}(m)=
 \begin{cases}
 (1+\frac{\epsilon}{4}-\epsilon\mu R_x)\sin(nx)+\frac{1}{\sqrt{2}}\sin((n+1)x) & \text{if } m=2n  \\
 (\mu R_x+\frac{1}{4}-\epsilon)\sin((n+1)x)-\frac{\epsilon}{\sqrt{2}}\sin(nx) & \text{if } m=2n+1,
\end{cases}
\\
& f(\epsilon,m)=f_\epsilon(m)=
\begin{cases}
 \frac{-1}{\sqrt{6}}(-\epsilon\sqrt{2})^{-n} & \text{if } m=2n \\
 \frac{1}{2\sqrt{3}}(-\epsilon\sqrt{2})^{-n} & \text{if } m=2n+1,
\end{cases}
\end{align*}
and for $m<0$, by
$$
f(j,m)=\sigma(j)f(j,-m-1).
$$
With Propositions \ref{spectrum_L^pr}, \ref{scalar_products1} and \ref{completeness1}, we have proven the following spectral theorem for the $\pi$-projected Laplacian $\mathcal{L}^{pr}$.
\begin{theorem}[Spectral theorem for the $\pi$-projected Laplacian] \label{spectral_thm_1}
Let $\{e_n\}_{n\in\Z}$ be the canonical basis of $\ell^2(\Z)$.
The linear map $U: \ell^2(\Z) \to L^2(J)$
 defined by 
 $$
(Ue_n)(j)=f(j,n)
$$
 is a unitary equivalence such that 
 $$
 U\mathcal{L}^{pr}U^{-1}=M_{\lambda},
 $$
 where 
 $\mathcal{L}^{pr}$ is the $\pi$-projected Laplacian on $\Z$ and $M_{\lambda}$ is the multiplication operator by the function $\lambda$, defined by its action on functions $g\in L^2(J)$, as
 $$
 (M_{\lambda}g)(j)=\lambda(j)g(j).
 $$
\end{theorem}

\begin{corollary}
 The spectrum of the $\pi$-projected Laplacian $\mathcal{L}^{pr}$ is given by
$$
\operatorname{Sp}(\mathcal{L}^{pr})= I_0 \sqcup\Big\{\frac{3}{4}\Big\}\sqcup I_1 \sqcup\Big\{\frac{7}{4}\Big\}
$$ 
where 
$$
I_0=\Big[\lambda_0,\frac74-\lambda_1\Big], \quad I_1=\Big[\lambda_1,\frac74-\lambda_0\Big],
$$
with
$$
\lambda_0=\frac78-\frac12\sqrt{\frac{25}{16}+\sqrt{2}}=0.01234\dots,\quad \lambda_1=\frac78+\frac12\sqrt{\frac{25}{16}-\sqrt{2}}=1.0675\dots.
$$
\end{corollary}

\begin{proposition} \label{prop:hk_proj}
The heat operator of the $\pi$-projected Laplacian $\mathcal{L}^{pr}$ of $L_\infty$ is given by 
 $$
 h_t^{pr}=e^{-t\mathcal{L}^{pr}}=\int_{\R} e^{-t\lambda} \mathrm{d}\nu^{pr}(\lambda),
 $$
 where $\nu^{pr}$ is the spectral measure associated to the $\pi$-projected Laplacian $\mathcal{L}^{pr}$ in $\ell^2(\Z)$. 
 \\
 Explicitly, $\nu^{pr}$ is a projection-valued measure defined on the Borel $\sigma$-algebra on $\R$, that is, $\nu^{pr}(\lambda):=\nu^{pr}((-\infty,\lambda])$ is a self-adjoint projection operator in $\ell^2(\Z)$ acting as 
 $$
 (\nu^{pr}(\lambda)g)(m)=\sum_{n\in \Z} \nu^{pr}(\lambda)(m,n)g(n), \quad g\in \ell^2(\Z),
 $$
 where the matrix coefficients are defined as 
\begin{multline}\label{proj_val_nu}
\nu^{pr}(\lambda)(m,n) =  \chi_{(-\infty,\lambda]} \Big(\frac 34 \Big) f(1,n) f(1,m) + \chi_{(-\infty,\lambda]} \Big(\frac 74 \Big) f(-1,n) f(-1,m) \\
   + \int_0^\pi \chi_{(-\infty,\lambda]\cap I_0} \Big( \frac{7}{8}-\frac{R_x}{2} \Big) \sum_{\epsilon\in\{\pm1\}} f(x,\epsilon,1,n) f(x,\epsilon,1,m) \frac{\mathrm{d}x}{H_{\epsilon}(x)} \\
 + \int_0^\pi \chi_{(-\infty,\lambda]\cap I_1} \Big( \frac{7}{8}+\frac{R_x}{2} \Big) \sum_{\epsilon\in\{\pm1\}} f(x,\epsilon,-1,n) f(x,\epsilon,-1,m) \frac{\mathrm{d}x}{H_{-\epsilon}(x)},
\end{multline}
where $H_\epsilon(x)$ is defined in~\eqref{H_eps}.
\end{proposition}

\begin{proof}
 From Theorem~\ref{spectral_thm_1}, we have 
 $$
 \mathcal{L}^{pr}=U^*M_{\lambda}U
 $$
 which implies that the heat kernel is given by 
 $$
 h_t^{pr}=e^{-t\mathcal{L}^{pr}}=U^*e^{-tM_{\lambda}}U, 
 $$
 so that its matrix coefficients are calculated as follows:
\begin{align*}
  h_t^{pr}(m,n)&=\langle e_m, h_t^{pr}e_n \rangle = \langle Ue_m, e^{-tM_{\lambda}}Ue_n \rangle = \int_J (Ue_m)(j) (e^{-tM_{\lambda}}Ue_n)(j)\mathrm{d}\eta(j) \\
  &=\int_J e^{-t\lambda(j)} f(j,m) f(j,n)\mathrm{d}\eta(j) \\
&= e^{-t\frac{3}{4}} f(1,m)f(1,n) +e^{-t\frac{7}{4}} f(-1,m)f(-1,n) \\
 & \quad + \int_0^\pi e^{-t(\frac{7}{8}-\frac{R_x}{2})} \sum_{\epsilon\in\{\pm1\}} f(x,1,\epsilon,m) f(x,1,\epsilon,n) \frac{\mathrm{d}x}{H_{\epsilon}(x)} \\
 & \quad + \int_0^\pi e^{-t(\frac{7}{8}+\frac{R_x}{2})} \sum_{\epsilon\in\{\pm1\}} f(x,-1,\epsilon,m) f(x,-1,\epsilon,n) \frac{\mathrm{d}x}{H_{-\epsilon}(x)} \\
 &=\int_{\R} e^{-t\lambda} \mathrm{d}\nu^{pr}(\lambda)(m,n),
\end{align*}
with $\nu^{pr}(\lambda)(m,n)$ defined in \eqref{proj_val_nu}.
The spectrum is given by the support of the integral.
\end{proof}

We now have all the preparations for proving Theorem~\ref{K1}.
\begin{proof}[Proof of Theorem~\ref{K1}]
Let $h_t=e^{-t\mathcal{L}}$ be the heat operator of the Laplacian $\mathcal{L}$ on $\Gamma$. 
Then, formula~\eqref{eq6} from Proposition~\ref{rel_covering_heat_kernel}, implies that the function $k_t(x)$ in~\eqref{heat_kernel_det}, is given by
$$
k_t(x)=K_t(\pi(x)), \quad K_t(n):=\sqrt{2}^{-\left\lceil\frac{n}{2}\right\rceil} h^{pr}_t(0,n),
$$
where we use the fact that $|\pi^{-1}(n)|=2^{\left\lceil\frac{n}{2}\right\rceil}$. More specifically, using Proposition~\ref{prop:hk_proj}, we obtain
\begin{align*}
 K_t(n) 
 = \ & \sqrt{2}^{-\left\lceil\frac{n}{2}\right\rceil} \Big( e^{-t\frac{3}{4}} f(1,n)f(1,0) +e^{-t\frac{7}{4}} f(-1,n)f(-1,0) \\
 & \quad + \int_0^\pi e^{-t(\frac{7}{8}-\frac{R_x}{2})} \sum_{\epsilon\in\{\pm1\}} f(x,1,\epsilon,n) f(x,1,\epsilon,0) \frac{\mathrm{d}x}{H_{\epsilon}(x)} \\
 & \quad + \int_0^\pi e^{-t(\frac{7}{8}+\frac{R_x}{2})} \sum_{\epsilon\in\{\pm1\}} f(x,-1,\epsilon,n) f(x,-1,\epsilon,0) \frac{\mathrm{d}x}{H_{-\epsilon}(x)} \Big),
\end{align*}
and the coefficients $\alpha_n, \beta_n$ and $\gamma_n^{\pm}(s)$ in~\eqref{eqK1} are thus given by
\begin{align*}
 & \alpha_n = \sqrt{2}^{-\left\lceil\frac{n}{2}\right\rceil} f(1,n)f(1,0) \\
 & \beta_n = \sqrt{2}^{-\left\lceil\frac{n}{2}\right\rceil} f(-1,n)f(-1,0) \\
 & \gamma_n^{\pm}(s) = \sqrt{2}^{-\left\lceil\frac{n}{2}\right\rceil} \sum_{\epsilon\in\{\pm1\}} f(s,\mp1,\epsilon,n) f(s,\mp1,\epsilon,0) \frac{1}{H_{\mp\epsilon}(s)}.
\end{align*}
\end{proof}

\begin{remark}
  A referee suggests there might be an alternative approach using Hecke algebras, more in line with Rivin's paper \cite{MR2587834}.
\end{remark}

\section*{Funding}
This work was supported by the Swiss NSF Grants [200020-200400, 200021-212864] and the Swedish Research Council Grant [104651320].

\section*{Acknowledgments} 
We thank Laurent Bartholdi, Pierre de la Harpe, Rostislav Grigorchuk, Jay Jorgenson and Tatiana Nagnibeda for valuable discussions, and Emmanuel Kowalski for pointing out reference \cite{MR4418463}. We are also grateful to Michael Magee for several insightful comments, in particular for predicting the existence of primes $p$ with spectral gap below the value $\lambda_0$.

\appendix
\section{Derivation of two distributional identities} \label{appA}

Let us show \eqref{2sin}. We have
\begin{align*}
 \sum_{n=0}^\infty 2\sin(nx)\sin(nx') &= \sum_{n=0}^\infty \big( \cos(n(x-x'))-\cos(n(x+x')) \big) \\
 &= \frac{1}{2} \sum_{n\in\Z} \big( \cos(n(x-x'))-\cos(n(x+x')) \big) \\ 
 &= \frac{1}{2} \operatorname{Re} \Big( \sum_{n\in\Z} \big(e^{in(x-x')}-e^{in(x+x')}\big) \Big).
\end{align*} 
Using \eqref{poisson_summation} and $\delta(\frac{x}{a})=a\delta(x)$, for $a\in \R_{>0}$, we have 
\begin{equation} \label{poisson1}
 \sum_{n\in\Z}e^{in(x\pm x')}=2\pi\sum_{k\in\Z}\delta(x\pm x'+2\pi k).
\end{equation}
This, together with the fact that $x,x'\in[0,\pi]$, gives
\begin{equation} \label{dist1}
 \sum_{n=0}^\infty 2\sin(nx)\sin(nx') = \pi \big( \sum_{k\in\Z}\delta(x-x'+2\pi k) + \sum_{k\in\Z}\delta(x+x'+2\pi k) \big)=\pi \delta(x-x'),
\end{equation}
where the second sum vanishes as a distribution over $[0,\pi]^2$.

Let us show \eqref{lastsum_final}. Using \eqref{2sin_id}, we have
\begin{equation}
  \sum_{n=0}^\infty \Big( 2\sin(nx)\sin((n+1)x') + 2\sin(nx')\sin((n+1)x) \Big) \label{lastsum_expr}
 \end{equation} 
 \begin{equation*}
  =  \sum_{n=0}^\infty \Big( \cos(n(x-x')-x')-\cos(n(x+x')+x') + \cos(n(x'-x)-x)-\cos(n(x+x')+x) \Big) \notag 
   \end{equation*} 
 \begin{equation*}
  =  \big(\cos(x')+\cos(x)\big) \sum_{n=0}^\infty \Big( \cos(n(x-x'))-\cos(n(x+x'))\Big) \notag 
   \end{equation*} 
 \begin{equation*}
  +  \sum_{n=0}^\infty \Big(\sin(n(x-x'))\big(\sin(x')-\sin(x) \big) +\sin(n(x+x'))\big(\sin(x')+\sin(x)\big) \Big) \label{lastsum}. 
\end{equation*}
The first sum in the obtained expression was already computed in \eqref{dist1}:
\begin{equation} \label{last_first_sum}
 \sum_{n=0}^\infty \Big( \cos(n(x-x'))-\cos(n(x+x')) \Big)= \sum_{n=0}^\infty 2\sin(nx)\sin(nx')=\pi \delta(x-x').
\end{equation}
Let us show that the second sum vanishes. Denoting $y=x+x'$ and $z=x-x'$, it becomes
\begin{align*}
& \sum_{n=0}^\infty \Big(\sin(nz)\big(\sin(\frac{y-z}{2})-\sin(\frac{y+z}{2}) \big) +\sin(ny)\big(\sin(\frac{y-z}{2})+\sin(\frac{y+z}{2})\big) \Big) \\
&= 2\sum_{n=0}^\infty \Big( \sin(ny) \sin(\frac y2)\cos(\frac z2) - \sin(nz)\sin(\frac z2)\cos(\frac y2) \Big)
\end{align*} 
where we used the sine of sum formula. Observing that the second term is the negative first term with the exchange $y\leftrightarrow z$, we only analyze the first term. Using \eqref{2sin_id}, we obtain
\begin{align*}
 \cos(\frac z2) \sum_{n=0}^\infty 2\sin(ny) \sin(\frac y2) &=\cos(\frac z2) \sum_{n=0}^\infty \Big( \cos((n-\frac12)y)- \cos((n+\frac12)y) \Big)\\
 &= \cos(\frac z2) \Big( \sum_{n=-1}^\infty \cos((n+\frac12)y)-\sum_{n=0}^\infty \cos((n+\frac12)y) \Big)=\cos(\frac z2)\cos(\frac y2),
\end{align*}
which symmetric in $y$ and $z$.
Finally, we conclude using \eqref{Diracdelta}
\begin{align*}
  \sum_{n=0}^\infty \Big( 2\sin(nx) & \sin((n+1)x') + 2\sin(nx')\sin((n+1)x) \Big) =2\pi \cos(x) \delta(x-x').
\end{align*}


\begin{thebibliography}{10}

\bibitem{MR2415383}
Bourgain, J. and Gamburd, A., 2008, Uniform expansion bounds for Cayley graphs of
\({\rm SL}_2(\mathbb F_p)\), Ann. of Math. (2), \underline{167}, 625--642.

\bibitem{MR830648}
Buck, M. W., 1986, Expanders and diffusers, SIAM J. Algebraic Discrete Methods,
\underline{7}, 282--304.

\bibitem{MR820357}
Cartwright, D. I. and Soardi, P. M., 1986, Harmonic analysis on the free product
of two cyclic groups, J. Funct. Anal., \underline{65}, 147--171.

\bibitem{MR3394421}
Chinta, G., Jorgenson, J. and Karlsson, A., 2015, Heat kernels on regular graphs
and generalized Ihara zeta function formulas, Monatsh. Math., \underline{178},
171--190.

\bibitem{MR1667452}
Chung, F. and Yau, S.-T., 1999, Coverings, heat kernels and spanning trees,
Electron. J. Combin., \underline{6}, Research Paper 12, 21 pp.

\bibitem{MR1653343}
Cowling, M., Meda, S. and Setti, A. G., 2000, Estimates for functions of the
Laplace operator on homogeneous trees, Trans. Amer. Math. Soc., \underline{352},
4271--4293.

\bibitem{MR4458570}
Grigorchuk, R., Nagnibeda, T. and P\'erez, A., 2022, On spectra and spectral
measures of Schreier and Cayley graphs, Int. Math. Res. Not. IMRN,
\underline{2022}, 11957--12002.

\bibitem{MR3348442}
Helfgott, H. A., 2015, Growth in groups: ideas and perspectives, Bull. Amer.
Math. Soc. (N.S.), \underline{52}, 357--413.

\bibitem{MR5005075}
Jorgenson, J., Karlsson, A. and Smajlovi\'c, L., 2026, Constructing heat kernels
on infinite graphs, Anal. Appl. (Singap.), \underline{24}, 169--198.

\bibitem{MR3144176}
Kowalski, E., 2013, Explicit growth and expansion for \({\rm SL}_2\), Int. Math.
Res. Not. IMRN, \underline{2013}, 5645--5708.

\bibitem{MR2569682}
Lubotzky, A., 2010, Discrete groups, expanding graphs and invariant measures,
Mod. Birkh\"auser Class., Birkh\"auser Verlag, Basel.

\bibitem{magee_optimal_spectral_gaps_2024}
Magee, M. R., 2024, Optimal spectral gaps, Institute for Advanced Study,
Princeton, lecture, video available at
\url{https://www.ias.edu/video/optimal-spectral-gaps}.

\bibitem{MR2587834}
Mohar, B. and Rivin, I., 2010, Simplices and spectra of graphs, Discrete Comput.
Geom., \underline{43}, 516--521.

\bibitem{MR3985838}
Rivin, I. and Sardari, N. T., 2019, Quantum chaos on random Cayley graphs of
\({\rm SL}_2[\mathbb Z/p\mathbb Z]\), Exp. Math., \underline{28}, 328--341.

\bibitem{MR4418463}
Rudnev, M. and Shkredov, I. D., 2022, On the growth rate in
\({\rm SL}_2(\mathbb F_p)\), the affine group and sum-product type implications,
Mathematika, \underline{68}, 738--783.

\bibitem{MR182610}
Selberg, A., 1965, On the estimation of Fourier coefficients of modular forms,
Proc. Sympos. Pure Math., \underline{8}, Amer. Math. Soc., Providence, RI,
1--15.

\bibitem{MR1954121}
Serre, J.-P., 2003, Trees, Springer Monogr. Math., Springer-Verlag, Berlin.

\end{thebibliography}

\end{document}